\documentclass[10pt CJK,oneside]{article}

\usepackage{amsfonts}
\usepackage{amsthm}
\usepackage{amsmath,amssymb}
\usepackage{multirow}
\usepackage{cite}
\usepackage{CJK}
\usepackage{color}
\usepackage[toc,page,title,titletoc,header]{appendix}

\numberwithin{equation}{section}
\newtheorem{theorem}{Theorem}
\newtheorem{lemma}{Lemma}[section]
\newtheorem{proposition}{Proposition}[section]
\newtheorem{definition}{Definition}

\newtheorem{corollary}{Corollary}[section]
\newtheorem{remark}{Remark}

\allowdisplaybreaks

\newcommand{\N}{{\mathbb{N}}}

\begin{document}
\renewcommand{\baselinestretch}{1.1}
\parskip 2pt
\large

\title{\bf \Large Invariant algebraic surfaces of the FitzHugh-Nagumo system}
\author {{\Large Liwei Zhang, Jiang Yu\thanks{Corresponding author. Research supported by the National Natural Science
Foundations of China (No.11431008) and NSF of Shanghai grant (No.15ZR1423700).}} \\
{Department of Mathematics, Shanghai Jiao Tong University}\\
{Shanghai  200240,  China}\\
\vspace{2mm} {E-mail: zhangliwei01@sjtu.edu.cn; jiangyu@sjtu.edu.cn;} }
\date{}
\maketitle
\begin{minipage}{14.5cm}

{\bf Abstract}: In this paper, we characterize all the irreducible Darboux polynomials and polynomial first integrals of FitzHugh-Nagumo (F-N) system. The method of the weight homogeneous polynomials and the characteristic curves is widely used to give a complete classification of Darboux polynomials of a system. However, this method does not work for F-N system. Here by considering the Darboux polynomials of an assistant system associated to F-N system, we classified the invariant algebraic surfaces of F-N system. Our results show that there is no invariant algebraic surface of F-N system in the biological parameters region.


{\bf AMS(2010) Subject Classification:} 37G05, 34C20, 34C14, 37J35.

\end{minipage}
\vspace{5mm}

\section{Introduction }

In this paper, we consider the FitzHugh-Nagumo (F-N) system
\begin{equation}\label{e0}
u_{t}=u_{xx}-f(u)-v, v_{t}=\varepsilon(u-\gamma v),
\end{equation}
where $f(u)=u(u-1)(u-a), $ and $0<a<\frac{1}{2}$, $\varepsilon>0,\gamma>0$ are biological parameters. We say, $u$ presents the voltage inside the axon at position $x\in\mathbb{R}$ and time $t$; $v$ presents a part of trans-membrance current that is  passing slowly adapting iron channels.

These equations were introduced in papers of FitzHugh \cite{17} and Nagumo et al. \cite{10}. FitzHugh\cite{17} simplified the 4-dimensional Hodgkin-Huxley (H-H) system into a planar system which is called Bonhoeffer-Van der Por system (BVP system for short). In \cite{17}, the author considered the excitable and oscillatory behavior of BVP system, and showed the underlying relationship between BVP system and H-H system. By the method that FitzHugh used in \cite{17} and the Kirchhoff's law, Nagumo et al. \cite{10} considered the propagation of the excitation along the nerve axon into H-H system, then H-H system becomes the partial differential system \eqref{e0}. From then on, FitzHugh-Nagumo (F-N) system \eqref{e0} has been studied extensively in the literature, and it becomes one of the simplest models describing the excitation of neural membranes and the propagation of nerve impulses along an axon. It has been attracting lots of attentions about the existence, uniqueness and stability of this equation's travelling wave, such as in \cite{1,6,7,8,10,18} and the references therein.

If we assume that the travelling wave of F-N system is a bounded solution $(u, v)(x, t)$ with $x, t\in\mathbb{R}$,  satisfying $(u, v)(x, t)=(u, v)(x+ct)$, where $c$ is a constant denoting the wave speed. Substituting $u=u(x+ct), v=v(x+ct)$ into system \eqref{e0}, then we have the following ordinary differential system
\begin{align}\label{e1}
\dot{x}&=z=P(x,y,z),\\
\dot{y}&=b(x-dy)=Q(x,y,z),\nonumber\\
\dot{z}&=x(x-1)(x-a)+y+cz=R(x,y,z),\nonumber
\end{align}
where the dot denotes derivative with respect to $\tau$ with $\tau=x+ct$, and $x=u, y=v, z=\dot{u}, b=\frac{\varepsilon}{c}, d=\gamma$.

In this paper, we focus on Darboux polynomials of 3-dimensional F-N system \eqref{e1}, with arbitrary parameters $a, b, c, d\in\mathbb{R}$. The Darboux polynomials and Darboux theory of integrability are considered as one of the important tools to look for first integral, and they also provide a relationship between the integrability of polynomial vector fields and the number of algebraic invariant surfaces of the system (see more details about the Darboux integrability in\cite{5, 19}). There have been lots of papers discussing a system's invariant algebraic surfaces (Darboux polynomials), such as \cite{2,3,14,15,20}. By studying the first integral or the invariant surface of a system,  we can make a more precise analysis of the topological structure of dynamics of the system.

 For F-N system \eqref{e1}, Jaume Llibre and Cl\`{a}udia Valls \cite{11} studied the system's analytic first integrals. When parameters $b=c=0$, they obtained the analytic first integrals, which is the same as the Darboux polynomials with zero cofactor we have obtained in this paper. For F-N system \eqref{e0},  if we don't consider the propagation of the excitation along the nerve axon,  i.e. $u_{xx}=b=$ constant, then it becomes the following planar F-N system
\begin{align*}
\dot{x}&=x(1-x)(x+a)-y+b,\\
\dot{y}&=d(x-cy),
\end{align*}
where $a, b, c, d\in\mathbb{R}$ are parameters. In \cite{12}, the authors studied the Liouvillian integrability of the planar F-N system. Because the Liouvillian integrability is equivalent to the Darboux integrability for the planar polynomial vector fields, they classified the invariant algebraic curves of planar F-N system.

 The method we used in this paper was introduced in \cite{5}. This method contains the use of weight homogeneous polynomials and the method of characteristic curves for solving linear partial differential equations. And this method has been widely used to deal with the invariant algebraic surfaces of many famous systems, such as Lorenz system \cite{13}, generalized Lorenz system \cite{16}, Muthuswamy-chua system \cite{15} et al.

 According to the characteristic curve method which will be introduced in the next section, the critical step is to construct linear partial differential operator in order to obtain a pair of characteristic curves. However, it is difficult to find a linear differential operator which is valid for system \eqref{e1}. To overcome this difficulty, we construct an assistant system \eqref{e24} associated to system \eqref{e1}. For system \eqref{e24}, we can find the linear differential operator to classify the Darboux polynomials by the characteristic curve method. Thus, based on the relationship between assistant system \eqref{e24} and system \eqref{e1}, we give a complete classification of the invariant algebraic surfaces of F-N system \eqref{e1}.

In the following, we provide some definitions which is necessary to our proof.
 \begin{definition}\label{d1}
 A polynomial $f(x, y, z)\in\mathbb{C}[x, y, z]$ and $f(x, y, z)\not\equiv 0$,  the ring of the complex coefficient polynomials in $x, y, z$ is called a Darboux polynomial for the  F-N system \eqref{e1} if
 $$
 \frac{\partial f}{\partial x}P+\frac{\partial f}{\partial y}Q+\frac{\partial f}{\partial z}R=kf
 $$
 for some real polynomial $k(x, y, z)$,  and $k(x, y, z)$ is called cofactor of $f$.
 \end{definition}
It is easy to know that the degree of $k(x, y, z)$ is smaller than the degree of system \eqref{e1},  where the degree of system \eqref{e1} equals to max degree of $\{P, Q, R\}$. If $f(x, y, z)$  is a Darboux polynomial, then the set $\{(x, y, z)\in\mathbb{R}^{3}\  |\ f(x, y, z)=0\}$ is invariant under the flow of system \eqref{e1}.

\begin{definition}
If $f(x, y, z)$  is a Darboux polynomial, the set $\{(x, y, z)\in\mathbb{R}^{3}\ |\ f(x, y, z)=0\}$ is called an invariant algebraic surface.
\end{definition}
We usually call simply $f(x, y, z)=0$ an invariant algebraic surface.

\begin{definition}
A real function $H(x, y, z, t):\mathbb{R}^{3}\times\mathbb{R}\rightarrow\mathbb{R}$ is called an invariant of system \eqref{e1},  if $H(x, y, z, t)=constant$ on all solution curves $(x(t), y(t), z(t))^T$ of system \eqref{e1}. If the invariant $H$ is independent of the time,  then it is called a first integral.
\end{definition}
If $H(x, y, z, t)$ is differentiable in $\mathbb{R}^{3}$,  then $H$ is an invariant of system \eqref{e1} if and only if $$\frac{\partial H}{\partial t}+\frac{\partial H}{\partial x}P+\frac{\partial H}{\partial y}Q+\frac{\partial H}{\partial z}R=0.$$
\begin{definition}
A function $H(x, y, z)$ is called an algebraic function, if it satisfies the algebraic equation
$$
f_{0}+f_{1}C+f_{2}C^{2}+\ldots+f_{n-1}C^{n-1}+C^{n}=0,
$$
where $f_{i}(x, y, z), (i=0, 1, \ldots, n-1)$ are given rational functions,  and $n$ is the smallest positive integer for which such a relation holds.
\end{definition}
Obviously,  any rational and polynomial functions are algebraic.
\begin{definition}
If the first integral $H$ is a polynomial (resp. a rational function or an algebraic function),  then it is called a polynomial first integral (resp. a rational first integral or an algebraic first integral).
\end{definition}
\begin{definition}
Let $H_{1}(x, y, z)$ and $H_{2}(x, y, z)$ be two first integrals of system \eqref{e1}. $H_{1}$ and $H_{2}$ are called independent, if their gradients are linearly independent vectors for all points $(x, y, z)\in\mathbb{R}^{3}$ except for a zero Lebesgue measure sets.
\end{definition}
The F-N system \eqref{e1} is called polynomial integrable, or rational integrable or algebraic integrable if it has two independent polynomial first integrals,  or rational first integrals or algebraic first integrals.
\begin{definition}
Let $\mathcal{S}$ be the set of all Darboux polynomials for system \eqref{e1}. A set $\mathcal{T}$ is called a minimum subset of $\mathcal{S}$, if every element of $\mathcal{S}$ is the finite product of the elements of $\mathcal{T}$, and the finite addition of the elements of $\mathcal{T}$ with the same cofactor. Furthermore we call every element of $\mathcal{T}$ a generator of $\mathcal{S}$.
\end{definition}
According to the Definition 7,  the generators for the set of a system's Darboux polynomials are irreducible Darboux polynomials.

The following proposition indicts the reason why we find all the irreducible Darboux polynomials.
\begin{proposition}
Assume that $f(x)\in\mathbb{C}[x]$ has an irreducible decomposition,  saying $f(x)=f_{1}^{l_{1}}\cdots f_{m}^{l_{m}}$ with $f_{i}\in\mathbb{C}[x]$ and $l_{i}\in\mathbb{N}$ for $i\in\{1, \ldots, m\}.$ Then $f(x)$ is a Darboux polynomial of a polynomial vector field if and only if $f_{i}$ for $i=1, \ldots, m$ are Darboux polynomials of the system. Moreover,  if $k(x)$ and $k_{i}(x)$ are cofactors of $f(x)$ and $f_{i}(x)$ for $i=1, \ldots, m$ respectively, then $k(x)=l_{1}k_{1}(x)+\cdots+l_{m}k_{m}(x)$.
\end{proposition}
In the following, we state the main results in this paper.
\begin{theorem}\label{t1}
The F-N system \eqref{e1} has six generators for the set of Darboux polynomials listed in the Table 1:
\begin{table}
\begin{center}
\begin{tabular}[t]{|l|c|l|}
\hline
Darboux polynomials & Cofactors & Parameters\\

\hline

$\frac{1}{2}x^{4}-z^{2}+2xy+\frac{2}{3}cxz+(\frac{1}{9}c^{2}-1)x^{2}$ & $\frac{4}{3}c$ & $a=-1, bd=-c, b=\frac{2}{27}c^{3}-\frac{1}{3}c,c\neq0$\\

\hline

$\frac{1}{2}x^{4}-z^{2}+2xy+\frac{2}{3}cxz+(\frac{1}{9}c^{2}-1)x^{2}-\frac{1}{2}dy^{2}$ & $\frac{4}{3}c$ & $a=-1, bd=-\frac{2}{3}c, b=\frac{2}{27}c^{3}-\frac{1}{3}c,c\neq0$ \\
\hline

$\frac{1}{2}x^{4}-z^{2}+2xy+\frac{2}{3}cxz-\frac{2}{3}(a+1)x^{3}$ & & $a\neq-1, bd=-c, c\neq0$,  \\
\quad $+(\frac{1}{9}c^{2}+a)x^{2}
-\frac{2}{9}c(a+1)z-\frac{2}{3}(a+1)y$\qquad\qquad & $\frac{4}{3}c$ & $b=\frac{2}{27}c^{3}-\frac{1}{9}a^{2}c+\frac{1}{9}ac-\frac{1}{9}c$\\

\quad $-\frac{2}{27}c^{2}(a+1)x$\qquad\qquad\qquad\qquad\qquad &  & $2c^{2}+3a^{2}-12a+3=0$\\

\hline

$\frac{1}{2}x^{4}-z^{2}-\frac{1}{2}dy^{2}+2xy+\frac{2}{3}cxz-\frac{2}{3}(a+1)x^{3}$ & & $a\neq-1, bd=-\frac{2}{3}c, c\neq0,$  \\
\quad $+(\frac{1}{9}c^{2}+a)x^{2}
-\frac{2}{9}c(a+1)z-\frac{1}{3}(a+1)y$\qquad\qquad & $\frac{4}{3}c$ & $b=\frac{2}{27}c^{3}-\frac{1}{9}a^{2}c+\frac{1}{9}ac-\frac{1}{9}c$\\

\quad $-\frac{2}{27}c^{2}(a+1)x$\qquad\qquad\qquad\qquad\qquad &  & $2c^{2}+a^{2}-7a+1=0$\\
\hline

\qquad\qquad\qquad\qquad\qquad$y$ & $0$ & $b=c=0$\\
\hline

$\frac{1}{4}x^{4}-\frac{1}{2}z^{2}-\frac{1}{3}(a+1)x^{3}+xy+\frac{1}{2}ax^{2}$ & $0$ & $b=c=0$\\
\hline
\end{tabular}
\end{center}
\caption{The generators for the set of Darboux polynomials of system \eqref{e1}}
\end{table}
\end{theorem}

\begin{corollary}\label{c1}
\begin{description}
\item{(a)}\ The F-N system \eqref{e1} has polynomial first integrals if and only if $b=c=0$.
\item{(b)}\ The F-N system \eqref{e1} is polynomial integrable.
\end{description}
\end{corollary}
This paper is  organised as follows. In Section 2,  we introduce the method of characteristic curves for solving linear partial differential equation and the weight homogeneous polynomials; in Section 3, we present the proof of Theorem \ref{t1} and Corollary \ref{c1}.

\begin{remark}
Theorem \ref{t1} implies that there is no invariant algebraic surface of F-N system in the biological parameters region $0<a<\frac{1}{2}$, $\varepsilon>0,\gamma>0$. Corollary \ref{c1} coincides with the results in \cite{11}.
\end{remark}

\section{The characteristic curve method }

The characteristic curves method for solving linear differential system see for instance\cite{4, 13}.

Consider the following first-order linear partial differential equation
\begin{equation}\label{e21}
a_{1}(x, y, z)A_{x}+a_{2}(x, y, z)A_{y}+a_{3}(x, y, z)A_{z}+a_{0}(x, y, z)A=f(x, y, z),
\end{equation}
where $A=A(x, y, z), a_{i}, i\in\{0, 1, 2, 3\}$ and $f$ are $C^{1}$ maps.
\begin{definition}
A curve $(x(t), y(t), z(t))$ in the $xyz$ space is called a characteristic curve for the equation \eqref{e21},  if the vector $(a_{1}(x_{0}, y_{0}, z_{0}), a_{2}(x_{0}, y_{0}, z_{0}), a_{3}(x_{0}, y_{0}, z_{0}))^T$ is tangent to the curve at each point $(x_{0}, y_{0}, z_{0})$ on the curve.
\end{definition}
Based on the definition,  a characteristic curve is a solution of the system
$$\frac{dx}{dt}=a_{1}(x, y, z), \ \frac{dy}{dt}=a_{2}(x, y, z),\  \frac{dz}{dt}=a_{3}(x, y, z).$$

Without loss of generality, assuming $a_{1}(x, y, z)\neq0$ locally,  then the above system is reduced to the following system
\begin{equation}\label{e18}
\frac{dy}{dx}=\frac{a_{2}(x, y, z)}{a_{1}(x, y, z)},\ \frac{dz}{dx}=\frac{a_{3}(x, y, z)}{a_{1}(x, y, z)}.
\end{equation}

The ordinary differential equation \eqref{e18} is known as the \emph{characteristic equation} of \eqref{e21}.

Suppose $g(x, y, z)=c_{1}, h(x, y, z)=c_{2}$ is a solution of equation \eqref{e18} in the implicit form,  where $c_{1}$ and $c_{2}$ are arbitrary constants. Taking the change of the variables
$$u=x,\ v=g(x, y, z),\ w=h(x, y, z), $$
then we have
$$x=u,\ y=p(u, v, w),\ z=q(u, v, w).$$
Hence the linear partial differential equation \eqref{e21} becomes the following ordinary differential equation in $u$ for fixed $v,w$,
\begin{equation}\label{e22}
\overline a_{1}(u, v, w)\overline A_{u}+\overline a_{0}(u, v, w)\overline A=\overline f(u, v, w),
\end{equation}
where $\overline a_{0},  \overline a_{1},  \overline A$ and $\overline f$ are $a_{0}, a_{1}, A, f$ written in terms of $u, v$ and $w$.

If $\overline A=\overline A(u, v, w)$ is a solution of \eqref{e22},  then by the variable transformation,  we get
$$A(x, y, z)=\overline A(x, g(x, y, z), h(x, y, z)),$$
which is a solution of the linear partial differential equation \eqref{e21}. Moreover,
 we can obtain the general solution of \eqref{e21} from the general solution of \eqref{e22} by the variable transformation.

At last,  we provide a definition of weight homogeneous polynomials,  which is used in the proof of our results.
\begin{definition}
A polynomial $f(x_{1}, x_{2}, \ldots, x_{n})$ is said to be weight homogeneous if there exist $(s_{1}, s_{2}, \ldots, s_{n})\in\mathbb{N}$ and $d\in\mathbb{N}$ such that for all $\alpha\in\mathbb{R}\backslash\{0\}$,
$$
f(\alpha^{s_{1}}x_{1}, \alpha^{s_{2}}x_{2}, \ldots, \alpha^{s_{n}}x_{n})=\alpha^{d}f(x_{1}, x_{2}, \ldots, x_{n}),
$$
where $d$ is the weight degree of the polynomial,  $(s_{1}, s_{2}, \ldots, s_{n})$ is the weight exponents of the polynomial.
\end{definition}

\section{The proof of Theorem \ref{t1} }

In this section,  we will present our proof of  Theorem \ref{t1} on the basis of characteristic curve method. In order to obtain a pair of characteristic curves, we need to construct a linear partial differential operator. However, we find that it is impossible to get some operator for F-N system \eqref{e1}. Instead, we consider an assistant system corresponding to system \eqref{e1} as follows
\begin{align}\label{e24}
\dot{x}&=z,\\
\dot{y}&=b(x-dy)+mxz,\nonumber\\
\dot{z}&=x(x-1)(x-a)+y+cz.\nonumber
\end{align}

The difference between F-N system \eqref{e1} and the assistant system \eqref{e24} is whether there is a term $mxz$ in $Q(x, y, z)$. By adding the term $mxz$,  we can construct the linear partial differential operator $L$ based on system \eqref{e24},
\begin{equation}\label{op}
L=z\frac{\partial}{\partial x}+mxz\frac{\partial}{\partial y}+x^{3}\frac{\partial}{\partial z}.
\end{equation}

We can check that its corresponding  characteristic curves are weight polynomials of weight exponents $(1, 2, 2)$. And Darboux polynomial of system \eqref{e24} can be expanded into a series of weight polynomials on the basis of the characteristic curves. Hence, the characteristic curve method works for system \eqref{e24}. It is obvious that the Darboux polynomials of system \eqref{e1} are the ones of system \eqref{e24} with $m=0$.

In the following, we use the method of characteristic curves to discuss the Darboux polynomials of the assistant system \eqref{e24}.

In order to obtain the linear partial differential operator,  we make a weight change of variables,  $(X, Y, Z, T)=(\alpha x, \alpha^{2}y,  \alpha^{2}z,  \alpha^{-1}t)$,  $\alpha\in\mathbb{R}\backslash\{0\}$. The assistant system \eqref{e24} becomes
 \begin{align}\label{e2}
X^{\prime}&=Z,\nonumber\\
Y^{\prime}&=-\alpha bdY+\alpha^{2}bX+mXZ,\\
Z^{\prime}&=X^{3}-\alpha[(a+1)X^{2}-Y-cZ]+\alpha^{2}aX,\nonumber
\end{align}
where the prime denotes the derivative with respect to T.

The relationship between Darboux polynomials of system \eqref{e24} and system \eqref{e2} will be shown in the following.

Suppose that $f(x, y, z)$ is the Darboux polynomial of the assistant system \eqref{e24} with cofactor $k(x, y, z)$. It is easy to prove that the degree of $k$ is less than or equal to 2. Therefore, we can assume that the cofactor is in the form$$
k(x, y, z)=k_{200}x^{2}+k_{110}xy+k_{011}yz+k_{101}xz+k_{020}y^{2}+k_{002}z^{2}+k_{100}x+k_{010}y+k_{001}z+k_{0}.$$

Set
    $$F(X, Y, Z)=\alpha^{l}f(\alpha^{-1}X, \alpha^{-2}Y, \alpha^{-2}Z),$$
    $$K(X, Y, Z)=\alpha^{h}k(\alpha^{-1}X, \alpha^{-2}Y, \alpha^{-2}Z).$$
Expanding $f(x, y, z)$ and $k(x, y, z)$ by the order of weight degree with weight exponents $(1, 2, 2)$, where $l, h$ are the highest weight degree in the expansion of $f(x, y, z)$ and $k(x, y, z)$,  respectively. It is easy to check that $h=4$. Then $F(X, Y, Z)$ is the Darboux polynomial of system \eqref{e2} with cofactor $\alpha^{-3}K(X, Y, Z)$, since
\begin{equation*}
\begin{array}{ll}
\frac{dF}{dT}|_{\eqref{e2}}&=\alpha^{l+1}k(\alpha^{-1}x, \alpha^{-2}y, \alpha^{-2}z)f(\alpha^{-1}x, \alpha^{-2}y, \alpha^{-2}z)\\
&=\alpha^{-3}K(X, Y, Z)F(X, Y, Z).
\end{array}
\end{equation*}
In the view of $f=F|_{\alpha=1}$, the Darboux polynomials of F-N system \eqref{e1} satisfy $f|_{m=0}=F|_{\alpha=1, m=0}$.

We remark that an arbitrary polynomial can be expanded into a series of weight homogeneous polynomials by the order of weight degree with the same weight exponents. Then, we can expand
\begin{equation}\label{fF}
F(X, Y, Z)=F_{0}(X, Y, Z)+\alpha F_{1}(X, Y, Z)+\alpha^{2}F_{2}(X, Y, Z)+\cdots+\alpha^{l}F_{l}(X, Y, Z),
\end{equation}
where $F_{j}$ is a weight homogeneous polynomial with weight exponents $(1, 2, 2)$, $j=0, 1, \ldots, l$,  and weight degree is  $l-j$. Noticing that $F$ is a Darboux polynomial of system \eqref{e2} with the cofactor $\alpha^{-3}K$, we have from Definition \ref{d1} of Darboux polynomial that
\begin{align*}
&z\sum\limits^{l}_{j=0}\alpha^{j}\frac{\partial{F_{j}}}{\partial x}+(mxz-\alpha bdy+\alpha^{2}bx)\sum\limits^{l}_{j=0}\alpha^{j}\frac{\partial{F_{j}}}{\partial y}\\
&\quad +[x^{3}-\alpha[(a+1)x^{2}-y-cz]+\alpha^{2}ax]\sum\limits^{l}_{j=0}\alpha^{j}\frac{\partial{F_{j}}}{\partial z}\\
=&[\alpha^{-3}(k_{020}y^{2}+k_{002}z^{2}+k_{011}yz)+
\alpha^{-2}(k_{110}xy+k_{101}xz)\\
&\quad +\alpha^{-1}(k_{200}x^{2}+k_{010}y+k_{001}z)+k_{1}x+\alpha k_{0}]\sum\limits^{l}_{j=0}\alpha^{j}F_{j}.
\end{align*}
For convenience, we use $x, y, z$ and $k_{1}$ instead of $X, Y, Z$ and $k_{100}$ here.

Comparing the terms with $\alpha^{-3}, \alpha^{-2}, \alpha^{-1}$,  we have $K(x, y, z)=\alpha^3 k_{1}x+\alpha^4 k_{0}$, that is, the cofactor in \eqref{e24} is $k(x,y,z)=k_{1}x+k_{0}$.

Equating the terms with $\alpha^{j}$ $(j=0, 1, \ldots, l)$,  we get the following first linear partial partial differential equations
\begin{equation}\label{e3}
\begin{array}{ll}
L[F_{0}]&=k_{1}xF_{0},\\
L[F_{1}]&=k_{1}xF_{1}+k_{0}F_{0}+bdy\frac{\partial{F_{0}}}{\partial{y}}+
[(a+1)x^{2}-y-cz]\frac{\partial{F_{0}}}{\partial{z}},\\
L[F_{j}]&=k_{1}xF_{j}+k_{0}F_{j-1}+bdy\frac{\partial{F_{j-1}}}{\partial{y}}+
[(a+1)x^{2}-y-cz]\frac{\partial{F_{j-1}}}{\partial{z}}\\
&\quad -bx\frac{\partial{F_{j-2}}}{\partial{y}}-ax\frac{\partial{F_{j-2}}}{\partial{z}},\qquad\hfill j=2, 3, \ldots, l+3,
\end{array}
\end{equation}
when $j>l,  F_{j}=0$, and
$L$ is the linear partial differential operator of the form \eqref{op}.

For equations in \eqref{e3}, the characteristic equations are
$$\frac{dy}{dx}=mx,\ \frac{dz}{dx}=\frac{x^{3}}{z}.$$
The solutions of characteristic equations are
\begin{equation}\label{E3}
y-\frac{1}{2}mx^{2}=c_{1}, \quad \frac{1}{4}x^{4}-\frac{1}{2}z^{2}=c_{2},
\end{equation}
where $c_{1}$ and $c_{2}$ are constants of integration.

According to the method of characteristics curve,  we make the change of the variables
\begin{equation}\label{e4}
u=x,\ v=y-\frac{1}{2}mx^{2},\ w=\frac{1}{4}x^{4}-\frac{1}{2}z^{2}.
\end{equation}
Then, we have
$$
x=u,\ y=v+\frac{1}{2}mu^{2},\ z=\pm\sqrt{\frac{1}{2}u^{4}-2w}.
$$
In the following,  without loss of generality we take
\begin{equation}\label{e5}
x=u,\ y=v+\frac{1}{2}mu^{2},\ z=\sqrt{\frac{1}{2}u^{4}-2w}.
\end{equation}

Through the change of variables \eqref{e4} and \eqref{e5}, it follows from the first equation of \eqref{e3} for fixed $v$ and $w$ that
\begin{equation}\label{e6}
\sqrt{\frac{1}{2}u^{4}-2w}\frac{d\overline{F_{0}}}{du}=k_{1}u\overline F_{0},
\end{equation}
where $\overline{F_{0}}$ is the function $F_{0}$ written in the variables $u, v$ and $w$. In the following,  $\overline{F_{j}}$ denote $F_{j}$ written in $u, v$ and $w$. Solving Equation \eqref{e6}, we have
$$
\overline F_{0}(u, v, w)=\overline G_{0}(v, w)(\sqrt{2}u^{2}+2\sqrt{\frac{1}{2}u^{4}-2w})^{\frac{\sqrt{2}}{2}k_{1}},
$$
where $\overline G_{0}(v, w)$ is an arbitrary smooth function. Then
$$
F_{0}(x, y, z)=\overline F_{0}(u, v, w)=\overline F_{0}(x, y-\frac{1}{2}mx^{2}, \frac{1}{4}x^{4}-\frac{1}{2}z^{2}),
$$
is the solution of the first equation in \eqref{e3}. Similarly,
$$
F_{j}(x, y, z)=\overline F_{j}(x, y-\frac{1}{2}mx^{2}, \frac{1}{4}x^{4}-\frac{1}{2}z^{2}),\ j=1, \ldots, l,
$$
are the solutions of the other equations in \eqref{e3}.

Hence we write $F_0$ in variables $x,y$ and $z$ as follows
\begin{equation}\label{F0}
F_{0}(x, y, z)=\overline G_{0}(y-\frac{1}{2}mx^{2}, \frac{1}{4}x^{4}-\frac{1}{2}z^{2})(\sqrt{2}x^{2}+2z)^{\frac{\sqrt{2}}{2}k_{1}}.
\end{equation}
From the transformation \eqref{e4},  we know that $v(x, y, z)$ and $w(x, y, z)$ are weight homogeneous polynomials of degree $2$ and $4$ with weight exponents $(1, 2, 2)$. Then we know from \eqref{F0} that $G_{0}$ should be a weight polynomial of the following two forms for some $n\in \N$
\begin{equation}\label{e7}
G_{0}(x, y, z)=\overline G_{0}(v, w)=\sum\limits^{n}_{i=1}a_{i}(y-\frac{1}{2}mx^{2})^{2i-1}(\frac{1}{4}x^{4}-\frac{1}{2}z^{2})^{n-i},
\end{equation}
or
\begin{equation}\label{e8}
G_{0}(x, y, z)=\overline G_{0}(v, w)=\sum\limits^{n}_{i=0}a_{i}(y-\frac{1}{2}mx^{2})^{2i}(\frac{1}{4}x^{4}-\frac{1}{2}z^{2})^{n-i}.
\end{equation}

In the following,  we divide our proof into three parts. In section 3.1, we prove that the cofactor of Darboux polynomials of system \eqref{e24} must be a constant, Then in section 3.2 we classify the Darboux polynomial when the cofactor is a nonzero constant. In section 3.3,  we continue to discuss the Darboux polynomials with zero cofactor. In each section,  we shall consider the two different forms of $G_{0}$ respectively.

\subsection{The nonzero cofactor of Darboux polynomials}

In this subsection, we shall prove the following lemma.

\begin{lemma}\label{l6}
If $f(x,y,z)$ is a Darboux polynomial with nonzero cofactor of \eqref{e24}, then the cofactor is a constant, that is,
$k(x,y,z)=k_0$ and $k_{1}=0$.
\end{lemma}

\begin{proof}
First, take $F_0$ with $G_0$ in \eqref{e7}, i.e.
$$
F_{0}=\sum\limits^{n}_{i=1}a_{i}(y-\frac{1}{2}mx^{2})^{2i-1}(\frac{1}{4}x^{4}
-\frac{1}{2}z^{2})^{n-i}(\sqrt{2}x^{2}+2z)^{\frac{\sqrt{2}}{2}k_{1}}.
$$

Substituting $F_{0}$ into the second equation of equation of \eqref{e3},  we obtain that
\begin{align*}
L[F_{1}]
&=k_{1}xF_{1} \\
&+\sum\limits^{n}_{i=1}k_{0}a_{i}(y-\frac{1}{2}mx^{2})^{2i-1}(\frac{1}{4}x^{4}-\frac{1}{2}z^{2})^{n-i}(\sqrt{2}x^{2}+2z)^{\frac{\sqrt{2}}{2}k_{1}} \\
&+\sum\limits^{n}_{i=1}bd(2i-1)a_{i}(y-\frac{1}{2}mx^{2})^{2i-1}(\frac{1}{4}x^{4}-\frac{1}{2}z^{2})^{n-i}(\sqrt{2}x^{2}+2z)^{\frac{\sqrt{2}}{2}k_{1}}\\
&+\sum\limits^{n}_{i=1}\frac{1}{2}mbd(2i-1)a_{i}(y-\frac{1}{2}mx^{2})^{2i-2}(\frac{1}{4}x^{4}-\frac{1}{2}z^{2})^{n-i}(\sqrt{2}x^{2}+2z)^{\frac{\sqrt{2}}{2}k_{1}}x^{2}\\
&+\sum\limits^{n}_{i=1}\sqrt{2}(a+1)k_{1}a_{i}(y-\frac{1}{2}mx^{2})^{2i-1}(\frac{1}{4}x^{4}-\frac{1}{2}z^{2})^{n-i}(\sqrt{2}x^{2}+2z)^{\frac{\sqrt{2}}{2}k_{1}-1}x^{2}\\
&-\sum\limits^{n}_{i=1}(a+1)(n-i)a_{i}(y-\frac{1}{2}mx^{2})^{2i-1}(\frac{1}{4}x^{4}-\frac{1}{2}z^{2})^{n-i-1}(\sqrt{2}x^{2}+2z)^{\frac{\sqrt{2}}{2}k_{1}}x^{2}z\\
&-\sum\limits^{n}_{i=1}\sqrt{2}k_{1}a_{i}(y-\frac{1}{2}mx^{2})^{2i}(\frac{1}{4}x^{4}-\frac{1}{2}z^{2})^{n-i}(\sqrt{2}x^{2}+2z)^{\frac{\sqrt{2}}{2}k_{1}-1}\\
&+\sum\limits^{n}_{i=1}(n-i)a_{i}(y-\frac{1}{2}mx^{2})^{2i}(\frac{1}{4}x^{4}-\frac{1}{2}z^{2})^{n-i-1}(\sqrt{2}x^{2}+2z)^{\frac{\sqrt{2}}{2}k_{1}}z\\
&+\sum\limits^{n}_{i=1}\frac{1}{2}m(n-i)a_{i}(y-\frac{1}{2}mx^{2})^{2i-1}(\frac{1}{4}x^{4}-\frac{1}{2}z^{2})^{n-i-1}(\sqrt{2}x^{2}+2z)^{\frac{\sqrt{2}}{2}k_{1}}x^{2}z\\
&-\sum\limits^{n}_{i=1}\frac{\sqrt{2}}{2}mk_{1}a_{i}(y-\frac{1}{2}mx^{2})^{2i-1}(\frac{1}{4}x^{4}-\frac{1}{2}z^{2})^{n-i}(\sqrt{2}x^{2}+2z)^{\frac{\sqrt{2}}{2}k_{1}-1}x^{2}\\
&-\sum\limits^{n}_{i=1}\sqrt{2}k_{1}ca_{i}(y-\frac{1}{2}mx^{2})^{2i-1}(\frac{1}{4}x^{4}-\frac{1}{2}z^{2})^{n-i}(\sqrt{2}x^{2}+2z)^{\frac{\sqrt{2}}{2}k_{1}-1}z\\
&+\sum\limits^{n}_{i=1}c(n-i)a_{i}(y-\frac{1}{2}mx^{2})^{2i-1}(\frac{1}{4}x^{4}-\frac{1}{2}z^{2})^{n-i-1}(\sqrt{2}x^{2}+2z)^{\frac{\sqrt{2}}{2}k_{1}}z^{2}.
\end{align*}

By the transformations \eqref{e4} and \eqref{e5},   we obtain the linear ordinary differential equation of $\overline F_{1}$ with respect to $u$ for fixed $v$ and $w$
\begin{align*}
\frac{d\overline F_{1}}{du}
&=k_{1}\overline F_{1}\frac{u}{\sqrt{\frac{1}{2}u^{4}-2w}} \\
&+\sum\limits^{n}_{i=1}k_{0}a_{i}v^{2i-1}w^{n-i}(\sqrt{2}u^{2}+2\sqrt{\frac{1}{2}u^{4}-2w})^{\frac{\sqrt{2}}{2}k_{1}}\frac{1}{\sqrt{\frac{1}{2}u^{4}-2w}}\\
&+\sum\limits^{n}_{i=1}bd(2i-1)a_{i}v^{2i-1}w^{n-i}(\sqrt{2}u^{2}+2\sqrt{\frac{1}{2}u^{4}-2w})^{\frac{\sqrt{2}}{2}k_{1}}\frac{1}{\sqrt{\frac{1}{2}u^{4}-2w}}\\
&+\sum\limits^{n}_{i=1}\frac{1}{2}mbd(2i-1)a_{i}v^{2i-2}w^{n-i}(\sqrt{2}u^{2}+2\sqrt{\frac{1}{2}u^{4}-2w})^{\frac{\sqrt{2}}{2}k_{1}}\frac{u^{2}}{\sqrt{\frac{1}{2}u^{4}-2w}}\\
&-\sum\limits^{n}_{i=1}\sqrt{2}k_{1}a_{i}v^{2i}w^{n-i}(\sqrt{2}u^{2}+2\sqrt{\frac{1}{2}u^{4}-2w})^{\frac{\sqrt{2}}{2}k_{1}-1}\frac{1}{\sqrt{\frac{1}{2}u^{4}-2w}}\\
&+\sum\limits^{n}_{i=1}\sqrt{2}(a+1)k_{1}a_{i}v^{2i-1}w^{n-i}(\sqrt{2}u^{2}+2\sqrt{\frac{1}{2}u^{4}-2w})^{\frac{\sqrt{2}}{2}k_{1}-1}\frac{u^{2}}{\sqrt{\frac{1}{2}u^{4}-2w}}\\
&-\sum\limits^{n}_{i=1}\frac{\sqrt{2}}{2}mk_{1}a_{i}v^{2i-1}w^{n-i}(\sqrt{2}u^{2}+2\sqrt{\frac{1}{2}u^{4}-2w})^{\frac{\sqrt{2}}{2}k_{1}-1}\frac{u^{2}}{\sqrt{\frac{1}{2}u^{4}-2w}}\\
&+\sum\limits^{n}_{i=1}\frac{1}{2}m(n-i)a_{i}v^{2i-1}w^{n-i-1}(\sqrt{2}u^{2}+2\sqrt{\frac{1}{2}u^{4}-2w})^{\frac{\sqrt{2}}{2}k_{1}}u^{2}\\
&-\sum\limits^{n}_{i=1}(a+1)(n-i)a_{i}v^{2i-1}w^{n-i-1}(\sqrt{2}u^{2}+2\sqrt{\frac{1}{2}u^{4}-2w})^{\frac{\sqrt{2}}{2}k_{1}}u^{2}\\
&+\sum\limits^{n}_{i=1}(n-i)a_{i}v^{2i}w^{n-i-1}(\sqrt{2}u^{2}+2\sqrt{\frac{1}{2}u^{4}-2w})^{\frac{\sqrt{2}}{2}k_{1}}\\
&-\sum\limits^{n}_{i=1}\sqrt{2}k_{1}ca_{i}v^{2i-1}w^{n-i}(\sqrt{2}u^{2}+2\sqrt{\frac{1}{2}u^{4}-2w})^{\frac{\sqrt{2}}{2}k_{1}-1}\\
&+\sum\limits^{n}_{i=1}(n-i)ca_{i}v^{2i-1}w^{n-i-1}(\sqrt{2}u^{2}+2\sqrt{\frac{1}{2}u^{4}-2w})^{\frac{\sqrt{2}}{2}k_{1}}\sqrt{\frac{1}{2}u^{4}-2w}.
\end{align*}

Integrating this equation,  we have
\begin{align*}
\overline F_{1}&=(\sqrt{2}u^{2}+2\sqrt{\frac{1}{2}u^{4}-2w})^{\frac{\sqrt{2}}{2}k_{1}}(C\\
&+\sum\limits^{n}_{i=1}k_{0}a_{i}v^{2i-1}w^{n-i}\int\frac{1}{\sqrt{\frac{1}{2}u^{4}-2w}}du\\
&+\sum\limits^{n}_{i=1}bd(2i-1)a_{i}v^{2i-1}w^{n-i}\int\frac{1}{\sqrt{\frac{1}{2}u^{4}-2w}}du\\
&+\sum\limits^{n}_{i=1}\frac{1}{2}mbd(2i-1)a_{i}v^{2i-2}w^{n-i}\int\frac{u^{2}}{\sqrt{\frac{1}{2}u^{4}-2w}}du\\
&-\sum\limits^{n}_{i=1}\sqrt{2}k_{1}a_{i}v^{2i}w^{n-i}\int(\sqrt{2}u^{2}+2\sqrt{\frac{1}{2}u^{4}-2w})^{-1}\frac{1}{\sqrt{\frac{1}{2}u^{4}-2w}}du\\
&+\sum\limits^{n}_{i=1}(\sqrt{2}(a+1)-\frac{\sqrt{2}}{2}m)k_{1}a_{i}v^{2i-1}w^{n-i}\int(\sqrt{2}u^{2}+2\sqrt{\frac{1}{2}u^{4}-2w})^{-1}\frac{u^{2}}{\sqrt{\frac{1}{2}u^{4}-2w}}du\\
&-\sum\limits^{n}_{i=1}(\frac{1}{3}(a+1)-\frac{1}{6}m)(n-i)a_{i}v^{2i-1}w^{n-i-1}u^{3}\\
&+\sum\limits^{n}_{i=1}(n-i)a_{i}v^{2i}w^{n-i-1}u\\
&-\sum\limits^{n}_{i=1}\sqrt{2}k_{1}ca_{i}v^{2i-1}w^{n-i}\int(\sqrt{2}u^{2}+2\sqrt{\frac{1}{2}u^{4}-2w})^{-1}du\\
&+\sum\limits^{n}_{i=1}(n-i)ca_{i}v^{2i-1}w^{n-i-1}\int\sqrt{\frac{1}{2}u^{4}-2w}du.
\end{align*}

From the formulas in Appendix, every integrals in the expression of $\overline{F_{1}}$ can be expressed by
the combination of the two integrals $\int\frac{1}{\sqrt{\frac{1}{2}u^{4}-2w}}du$ and $\int\frac{u^{2}}{\sqrt{\frac{1}{2}u^{4}-2w}}du$. Since $F_{1}$ is a weight homogeneous polynomial, $\int\frac{1}{\sqrt{\frac{1}{2}u^{4}-2w}}du$ and $\int\frac{u^{2}}{\sqrt{\frac{1}{2}u^{4}-2w}}du$ are not polynomials,  then the coefficients of $\int\frac{1}{\sqrt{\frac{1}{2}u^{4}-2w}}du$ and $\int\frac{u^{2}}{\sqrt{\frac{1}{2}u^{4}-2w}}du$ should be vanish. Using the formulas in Appendix,  we can write the term with the integral $\int\frac{u^{2}}{\sqrt{\frac{1}{2}u^{4}-2w}}du$ in $\overline{F}_1$ as
$$\sum^n_{i=1}\frac{1}{4}k_{1}a_{i}v^{2i}w^{n-i-1}\int\frac{u^{2}}{\sqrt{\frac{1}{2}u^{4}-2w}}du=0, $$
which implies $k_{1}=0$.

Secondly, taking $F_0$ with $G_0$ in \eqref{e8}, we have
$$ F_{0}=\sum\limits^{n}_{i=1}a_{i}(y-\frac{1}{2}mx^{2})^{2i}(\frac{1}{4}x^{4}-
\frac{1}{2}z^{2})^{n-i}(\sqrt{2}x^{2}+2z)^{\frac{\sqrt{2}}{2}k_{1}}.
$$
By the same way which we used in the first part,  we can also prove $k_{1}=0$.
\end{proof}

\subsection{ Darboux polynomials with nonzero cofactor $k_{1}=0, k_{0}\neq0$}

In this subsection, we shall present the Darboux polynomials of system \eqref{e1} with nonzero cofactor and the conditions for their existence.

Now by Lemma \ref{l6} and \eqref{F0},  we have $F_{0}=G_{0}$,  i.e.
\begin{equation}\label{e9}
F_{0}=\sum\limits^{n}_{i=1}a_{i}(y-\frac{1}{2}mx^{2})^{2i-1}(\frac{1}{4}x^{4}-\frac{1}{2}z^{2})^{n-i},
\end{equation}
or
\begin{equation}\label{e10}
F_{0}=\sum\limits^{n}_{i=0}a_{i}(y-\frac{1}{2}mx^{2})^{2i}(\frac{1}{4}x^{4}-\frac{1}{2}z^{2})^{n-i}.
\end{equation}

Notice that the weight degree $l$ of $F_{0}$ in the form \eqref{e9} and \eqref{e10} is $4n-2$ and $4n$,  respectively.
\begin{remark}\label{r1}
We claim that not all of $a_i$ are zero, that is, $F_0\not\equiv 0$. Otherwise, in view of $k_1=0$, it follows from \eqref{e3} that $F_0\equiv 0$ implies $F\equiv 0$.
\end{remark}

\begin{lemma}\label{l2}
When $F_{0}$ is in the form \eqref{e9},  system \eqref{e24} doesn't have Darboux polynomial with nonzero cofactor.
\end{lemma}

\begin{proof}

Substituting $F_{0}$ of \eqref{e9} into the second equation of \eqref{e3},  we obtain that
\begin{align*}
L[F_{1}]=&\sum\limits^{n}_{i=1}(k_{0}+(2i-1)bd)a_{i}(y-\frac{1}{2}mx^{2})^{2i-1}(\frac{1}{4}x^{4}-\frac{1}{2}z^{2})^{n-i} \\
&+\sum\limits^{n}_{i=1}(\frac{1}{2}m-(a+1))(n-i)a_{i}(y-\frac{1}{2}mx^{2})^{2i-1}(\frac{1}{4}x^{4}-\frac{1}{2}z^{2})^{n-i-1}x^{2}z\\
&+\sum\limits^{n}_{i=1}(n-i)a_{i}(y-\frac{1}{2}mx^{2})^{2i}(\frac{1}{4}x^{4}-\frac{1}{2}z^{2})^{n-i-1}z\\
&+\sum\limits^{n}_{i=1}c(n-i)a_{i}(y-\frac{1}{2}mx^{2})^{2i-1}(\frac{1}{4}x^{4}-\frac{1}{2}z^{2})^{n-i-1}z^{2}\\
&+\sum\limits^{n}_{i=1}\frac{1}{2}mbd(2i-1)a_{i}(y-\frac{1}{2}mx^{2})^{2i-1}(\frac{1}{4}x^{4}-\frac{1}{2}z^{2})^{n-i}x^{2}
\end{align*}

Through the variable changes \eqref{e4} and \eqref{e5},   we obtain the following ordinary linear differential equation with respect to $u$ for fixed $v$ and $w$
\begin{align*}
\frac{d\overline F_{1}}{du}=&\sum\limits^{n}_{i=1}(k_{0}+(2i-1)bd)a_{i}v^{2i-1}w^{n-i}\frac{1}{\sqrt{\frac{1}{2}u^{4}-2w}} \\
&+\sum\limits^{n}_{i=1}(\frac{1}{2}m-(a+1))(n-i)a_{i}v^{2i-1}w^{n-i-1}u^{2}\\
&+\sum\limits^{n}_{i=1}(n-i)a_{i}v^{2i}w^{n-i-1}\\
&+\sum\limits^{n}_{i=1}c(n-i)a_{i}v^{2i-1}w^{n-i-1}\sqrt{\frac{1}{2}u^{4}-2w}\\
&+\sum\limits^{n}_{i=1}\frac{1}{2}mbd(2i-1)a_{i}v^{2i-1}w^{n-i}\frac{u^{2}}{\sqrt{\frac{1}{2}u^{4}-2w}}.
\end{align*}
Using formulas presented in the Appendix,  we solve this equation and obtain
\begin{align*}
\overline F_{1}=&\sum\limits^{n}_{i=1}(k_{0}+(2i-1)bd-\frac{4}{3}c(n-i))a_{i}v^{2i-1}w^{n-i}\int\frac{1}{\sqrt{\frac{1}{2}u^{4}-2w}}du \\
&+\sum\limits^{n}_{i=1}\frac{1}{2}mbd(2i-1)a_{i}v^{2i-1}w^{n-i}\int\frac{u^{2}}{\sqrt{\frac{1}{2}u^{4}-2w}}du\\
&-\sum\limits^{n}_{i=1}\frac{1}{3}(\frac{1}{2}m-(a+1))(n-i)a_{i}v^{2i-1}w^{n-i-1}u^{3}\\
&+\sum\limits^{n}_{i=1}(n-i)a_{i}v^{2i}w^{n-i-1}u\\
&+\sum\limits^{n}_{i=1}\frac{1}{3}c(n-i)a_{i}v^{2i-1}w^{n-i-1}u\sqrt{\frac{1}{2}u^{4}-2w}+\overline{G}_{1}(v, w),
\end{align*}
where $\overline{G}_{1}(v, w)$ is an arbitrary polynomial in $v$ and $w$. $G_1(x,y,z)=\overline{G}_{1}(v, w)$ is a weight polynomial with weight degree even. Noticing that $F_{1}$ is a weight homogeneous polynomial of weight degree $4n-3$,  then we have $\overline{G}_{1}(v, w)=0$,  and the coefficients of $\int\frac{1}{\sqrt{\frac{1}{2}u^{4}-2w}}du$ and $\int\frac{u^{2}}{\sqrt{\frac{1}{2}u^{4}-2w}}du$ must vanish. Then we have for $i=1,2, \ldots, n$,
\begin{equation}\label{e11}
(k_{0}+(2i-1)bd-\frac{4}{3}c(n-i))a_i=0 \quad mbda_{i}=0,i=1,2,\ldots,n,
\end{equation}
and
\begin{align*}
F_{1}=&\sum\limits^{n}_{i=1}\frac{1}{3}(\frac{1}{2}m-(a+1))(n-i)a_{i}(y-\frac{1}{2}mx^{2})^{2i-1}(\frac{1}{4}x^{4}-\frac{1}{2}z^{2})^{n-i-1}x^{3} \\
&+\sum\limits^{n}_{i=1}(n-i)a_{i}(y-\frac{1}{2}mx^{2})^{2i}(\frac{1}{4}x^{4}-\frac{1}{2}z^{2})^{n-i-1}x \\
&+\sum\limits^{n}_{i=1}\frac{1}{3}c(n-i)a_{i}(y-\frac{1}{2}mx^{2})^{2i-1}(\frac{1}{4}x^{4}-\frac{1}{2}z^{2})^{n-i-1}xz.
\end{align*}
Furthermore from Remark \ref{r1}, $F_0\not\equiv 0$ implies that there is at least $a_{i}\not=0$, then
$$
k_{0}=\frac{4}{3}c(n-i)-(2i-1)bd.
$$

Consider $F_2$ in \eqref{e3} with $j=2$. We observe that there is only a term $k_0F_1$ in \eqref{e3} with $j=2$, which includes $k_0$. Hence, we always have $k_0a_i$ after substituting $F_1$ into the equation. Without loss of generality, we can substitute $k_{0}=\frac{4}{3}c(n-i)-(2i-1)bd$, $i=1,2,\ldots,n$ into the equation. Then solving this ordinary differential equation with respect to $u$ for fixed $v$ and $w$,  we get
\begin{align*}
\overline F_{2}=&\sum\limits^{n}_{i=1}(\frac{4}{3}(n-i)c+bd-\frac{1}{3}c)(n-i)a_{i}v^{2i}w^{n-i-1}\int\frac{u}{\sqrt{\frac{1}{2}u^{4}-2w}}du \\
&-\sum\limits^{n}_{i=1}b(2i-1)a_{i}v^{2i-2}w^{n-i}\int\frac{u}{\sqrt{\frac{1}{2}u^{4}-2w}}du\\
&-\sum\limits^{n}_{i=1}\frac{4}{3}c(n-i)(n-i-1)a_{i}v^{2i}w^{n-i-1}\int\frac{u}{\sqrt{\frac{1}{2}u^{4}-2w}}du \\
&+\sum\limits^{n}_{i=1}\frac{1}{2}(\frac{4}{9}(n-i)c^{2}-\frac{1}{3}c^{2}+a)(n-i)a_{i}v^{2i-1}w^{n-i-1}u^{2}\\
&+\sum\limits^{n}_{i=1}\frac{1}{2}(n-i)(n-i-1)a_{i}v^{2i+1}w^{n-i-2}u^{2} \\
&-\sum\limits^{n}_{i=1}\frac{1}{3}(a+1)(n-i)(n-i-1)a_{i}v^{2i}w^{n-i-2}u^{4} \\
&+\sum\limits^{n}_{i=1}\frac{1}{6}m(n-i)(n-i-1)a_{i}v^{2i}w^{n-i-2}u^{4} \\
&+\sum\limits^{n}_{i=1}(\frac{1}{18}(a+1)^{2}-\frac{1}{18}m(a+1)+\frac{1}{72}m^{2})(n-i)(n-i-1)a_{i}v^{2i-1}w^{n-i-2}u^{6} \\
&+\sum\limits^{n}_{i=1}(\frac{4}{9}(n-i)+\frac{1}{3})c(a+1)(n-i)a_{i}v^{2i-1}w^{n-i-1}\sqrt{\frac{1}{2}u^{4}-2w} \\
&+\sum\limits^{n}_{i=1}(\frac{2}{9}(n-i)-\frac{1}{6})mc(n-i)a_{i}v^{2i-1}w^{n-i-1}\sqrt{\frac{1}{2}u^{4}-2w} \\
&+\sum\limits^{n}_{i=1}\frac{1}{3}c(n-i)(n-i-1)a_{i}v^{2i}w^{n-i-2}u^{2}\sqrt{\frac{1}{2}u^{4}-2w} \\
&-\sum\limits^{n}_{i=1}(\frac{1}{18}m-\frac{1}{9}(a+1))c(n-i)(n-i-1)a_{i}v^{2i-1}w^{n-i-2}(u^{4}-4w)\sqrt{(\frac{1}{2}u^{4}-2w)} \\
&+\sum\limits^{n}_{i=1}\frac{1}{3}c^{2}(n-i)(n-i-1)a_{i}v^{2i-1}w^{n-i-2}(\frac{1}{12}u^{6}-wu^{2})+\overline{G}_{2}(v, w),
\end{align*}
where  $\overline{G}_{2}(v, w)$ is a polynomial.

Because $F_{2}$ is weight homogeneous polynomial of weight degree $4n-4$,  the coefficients of $\int\frac{u}{\sqrt{\frac{1}{2}u^{4}-2w}}du$ vanish. Collecting the coefficients of $\int\frac{u}{\sqrt{\frac{1}{2}u^{4}-2w}}du$ in the expression of $\overline{F_{2}}$,  we have
\begin{equation}\label{E0}
\sum\limits^{n-1}_{i=1}((bd+c)(n-i)a_{i}-b(2i+1)a_{i+1})v^{2i}w^{n-i-1}-ba_{1}w^{n-1}=0,
\end{equation}
which implies $ba_{1}=0$ and
$$(bd+c)(n-i)a_{i}-b(2i+1)a_{i+1}=0,\ \ i=1,\ldots,n-1.$$

If $b\neq0$, then $a_{1}=0$, it indicts $a_{i}=0,i=1,\ldots,n$. It contradicts with Remark \ref{r1}.

If $b=0$, \eqref{E0} becomes
$$\sum\limits^{n-1}_{i=1}c(n-i)a_{i}v^{2i}w^{n-i-1}=0.$$
Then $c=0$, it contradicts with \eqref{e11} with $k_{0}\not=0.$
\end{proof}

From Lemma \ref{l2}, we shall take $F_0$ in \eqref{e10} for Darboux polynomials with  nonzero cofactor.

\begin{remark}\label{r2}
 It follows from Lemma \ref{l2} that the highest weight degree in the expansion \eqref{fF} of the Darboux polynomial $f$ with nonzero cofactor should be $4n,n\in\mathbb{N}$.
\end{remark}

\begin{lemma}\label{l3'}
If $F_{0}$ is in the form \eqref{e10}, parameters in system \eqref{e1} having Darboux polynomials with nonzero cofactors satisfy one of the following conditions:
\begin{description}
\item{(1)} $a=-1, bd=-c, b=\frac{2}{27}c^{3}-\frac{1}{3}c,c\neq 0$;
\item{(2)}$-\frac{1}{81}c^{2}-\frac{1}{27}a^{2}+\frac{4}{27}a-\frac{1}{27}=0,  bd=-c,  b=\frac{2}{27}c^{3}-\frac{1}{9}a^{2}c+\frac{1}{9}ac-\frac{1}{9}c, c\neq 0$;
\item{(3)} $a=-1, bd=-\frac{2}{3}c, b=\frac{2}{27}c^{3}-\frac{1}{3}c,c\neq 0$;
\item{(4)}$-\frac{1}{81}c^{2}-\frac{1}{27}a^{2}+\frac{4}{27}a-\frac{1}{27}=0,  bd=-\frac{2}{3}c,  b=\frac{2}{27}c^{3}-\frac{1}{9}a^{2}c+\frac{1}{9}ac-\frac{1}{9}c, c\neq 0$.
\end{description}
\end{lemma}

\begin{proof}

Substituting $F_0$ of \eqref{e10} into the second equation of \eqref{e3}, then solving it,  we get the following equations by using formulas in the Appendix
\begin{align}\label{e23}
\overline F_{1}=&\sum\limits^{n}_{i=0}(k_{0}+2ibd-\frac{4}{3}(n-i)c)a_{i}v^{2i}w^{n-i}\int\frac{1}{\sqrt{\frac{1}{2}u^{4}-2w}}du\nonumber\\
&+\sum\limits^{n}_{i=0}mbdia_{i}v^{2i-1}w^{n-i}\int\frac{u^{2}}{\sqrt{\frac{1}{2}u^{4}-2w}}du\nonumber\\
&-\sum\limits^{n}_{i=0}\frac{1}{3}(a+1)(n-i)a_{i}v^{2i}w^{n-i-1}u^{3}\\
&+\sum\limits^{n}_{i=0}(n-i)a_{i}v^{2i+1}w^{n-i-1}u\nonumber\\
&+\sum\limits^{n}_{i=0}\frac{1}{3}c(n-i)a_{i}v^{2i}w^{n-i-1}u\sqrt{\frac{1}{2}u^{4}-2w}\nonumber\\
&+\sum\limits^{n}_{i=0}\frac{1}{6}m(n-i)a_{i}v^{2i}w^{n-i-1}u^{3}+\overline{G}_{1}(v, w),\nonumber
\end{align}
where $\overline{G}_{1}(v, w)$ is an arbitrary polynomial in $v$ and $w$. $G_1(x,y,z)=\overline{G}_{1}(v, w)$ is a weight polynomial of weight degree even. However, $F_{1}$ is a weight homogeneous polynomial of weight degree $4n-1$, then $\overline{G}_{1}(v, w)=0$,  and the coefficients of $\int\frac{1}{\sqrt{\frac{1}{2}u^{4}-2w}}du$ and $\int\frac{u^{2}}{\sqrt{\frac{1}{2}u^{4}-2w}}du$ must vanish. Then we have $i=0,1, \ldots, n$
\begin{equation}\label{e20}
(k_{0}+2ibd-\frac{4}{3}c(n-i))a_{i}=0, \quad imbda_{i}=0,
\end{equation}
and from \eqref{e23}
\begin{align}\label{e13}
F_{1}=&-\sum\limits^{n}_{i=0}\frac{1}{3}(a+1)(n-i)a_{i}(y-\frac{1}{2}mx^{2})^{2i}(\frac{1}{4}x^{4}-\frac{1}{2}z^{2})^{n-i-1}x^{3}\nonumber\\
&+\sum\limits^{n}_{i=0}(n-i)a_{i}(y-\frac{1}{2}mx^{2})^{2i+1}(\frac{1}{4}x^{4}-\frac{1}{2}z^{2})^{n-i-1}x \\
&+\sum\limits^{n}_{i=0}\frac{1}{3}c(n-i)a_{i}(y-\frac{1}{2}mx^{2})^{2i}(\frac{1}{4}x^{4}-\frac{1}{2}z^{2})^{n-i-1}xz\nonumber\\
&+\sum\limits^{n}_{i=0}\frac{1}{6}m(n-i)a_{i}(y-\frac{1}{2}mx^{2})^{2i}(\frac{1}{4}x^{4}-\frac{1}{2}z^{2})^{n-i-1}x^{3}.\nonumber
\end{align}
Due to Remark \ref{r1} and the first equation in \eqref{e20}, then
\begin{equation}\label{e12}
k_{0}=\frac{4}{3}(n-i)c-2ibd=\frac{4}{3}nc-2i(bd+\frac{2}{3}c).
\end{equation}
By the same method which has been used in the proof of Lemma \ref{l2}, solving $\overline F_{2}$ in \eqref{e3} with $j=2$,  we get
\begin{align*}
\overline F_{2}=&\sum\limits^{n}_{i=0}(\frac{4}{3}(n-i)c+bd-\frac{1}{3}c)(n-i)a_{i}v^{2i+1}w^{n-i-1}\int\frac{u}{\sqrt{\frac{1}{2}u^{4}-2w}}du \\
&-\sum\limits^{n}_{i=0}2bia_{i}v^{2i-1}w^{n-i}\int\frac{u}{\sqrt{\frac{1}{2}u^{4}-2w}}du\\
&-\sum\limits^{n}_{i=0}\frac{4}{3}c(n-i)(n-i-1)a_{i}v^{2i+1}w^{n-i-1}\int\frac{u}{\sqrt{\frac{1}{2}u^{4}-2w}}du \\
&+\sum\limits^{n}_{i=0}\frac{1}{2}(\frac{1}{3}k_{0}c+\frac{2}{3}bdic-\frac{1}{3}c^{2}+a)(n-i)a_{i}v^{2i}w^{n-i-1}u^{2}\\
&+\sum\limits^{n}_{i=0}\frac{1}{2}(n-i)(n-i-1)a_{i}v^{2i+2}w^{n-i-2}u^{2} \\
&-\sum\limits^{n}_{i=0}(\frac{1}{3}(a+1)+\frac{1}{8}m)(n-i)(n-i-1)a_{i}v^{2i+1}w^{n-i-2}u^{4} \\
&+\sum\limits^{n}_{i=0}(\frac{1}{18}(a+1)^{2}-\frac{1}{36}m(a+1)+\frac{1}{72}m^{2})(n-i)(n-i-1)a_{i}v^{2i}w^{n-i-2}u^{6} \\
&+\sum\limits^{n}_{i=0}(\frac{1}{3}(-\frac{4}{3}(n-i)+1)(a+1)-\frac{1}{6}m)c(n-i)a_{i}v^{2i}w^{n-i-1}\sqrt{\frac{1}{2}u^{4}-2w}\\
&-\sum\limits^{n}_{i=0}\frac{1}{9}c(a+1)(n-i)(n-i-1)a_{i}v^{2i}w^{n-i-2}(u^{4}-4w)\sqrt{\frac{1}{2}u^{4}-2w} \\
&+\sum\limits^{n}_{i=0}\frac{1}{3}c(n-i)(n-i-1)a_{i}v^{2i+1}w^{n-i-2}u^{2}\sqrt{\frac{1}{2}u^{4}-2w}\\
&+\sum\limits^{n}_{i=0}\frac{1}{3}c^{2}(n-i)(n-i-1)a_{i}v^{2i}w^{n-i-2}(\frac{1}{12}u^{6}-wu^{2})+\overline{G}_{2}(v, w),
\end{align*}
where  $\overline{G}_{2}(v, w)$ is a weight polynomial.
Since $F_{2}$ is a weight homogeneous polynomial of weight degree $4n-2$, the coefficients of $\int\frac{u}{\sqrt{\frac{1}{2}u^{4}-2w}}du$ must vanish. Adjusting the coefficients of $\int\frac{u}{\sqrt{\frac{1}{2}u^{4}-2w}}du$ in the expression of $\overline{F}_{2}$,  we have
$$\sum\limits^{n}_{i=0}[(bd+c)(n-i)a_{i}-2b(i+1)a_{i+1}]v^{2i}w^{n-i-1}=0,$$
where $a_{n+1}=0$. It implies that for $i=0,1, \ldots, n$,
\begin{equation}\label{e19}
(bd+c)(n-i)a_{i}-2b(i+1)a_{i+1}=0,
\end{equation}
where $a_{n+1}=0$. We consider \eqref{e19} in the two cases, (i) $bd=-c$, (ii) $bd\not=-c$.
\begin{description}
\item{(i)} For $bd=-c$, it follows from equation \eqref{e19} that  $2bia_{i}=0, i=1,2, \ldots, n$. If there exits some $i$ such that $a_{i}\neq0, i=1,2, \ldots, n$,  then $b=0, c=0$,  and $k_{0}=0$ from \eqref{e12}, which contradicts with the cofactor is nonzero. Combing with Remark \ref{r1}, we have $a_{i}=0, i=1,2, \ldots, n$, and $a_{0}\neq0$. In addition, $k_{0}=\frac{4}{3}nc,c\neq0$ from \eqref{e12}.
\item{(ii)} For $bd\neq -c$, if $b=0$ then $c\neq0$ and $c(n-i)a_{i}=0, i=0,1, \ldots, n-1$ from \eqref{e19}. If there exits $i_{0}\in\{0,1, \ldots, n-1\}$ such that $a_{i_{0}}\neq0$,  then $c=0$. It contradicts with $c\neq0$. Thus, $a_{i}=0, i=0,1, \ldots, n-1$, and $a_{n}\neq0$. However, in this condition $k_{0}=\frac{4}{3}(n-n)c=0$, so there exits no Darboux polynomial with nonzero cofactor when $b=0$. Therefore, $b\neq 0.$

 If $b\neq0$, it follows from \eqref{e19} that $a_{i}\neq0$, $i=0,1, \ldots, n$. Besides, by the fact that the cofactor $k_0$ of one Darboux polynomial is unique, it easy to get from \eqref{e12} that
 $$bd=-\frac{2}{3}c\ \mbox{and}\ k_{0}=\frac{4}{3}nc,$$
 which implies that $c\neq0$. In addition, it follows from  $ibdma_{i}=0$ in \eqref{e20} that $m=0$.
\end{description}

By the previous analysis, we know that if system \eqref{e24} have Darboux polynomials $f$ with nonzero cofactor $k_0$,  then there are two cases for $F_{0}$ corresponding to $f$
\begin{description}
\item{(i)} $a_{0}\neq0,  a_{i}=0, i=1, 2, \ldots, n$ and $bd=-c,k_{0}=\frac{4}{3}nc, c\neq 0$;
\item{(ii)}$a_{i}\neq0,  i=0, 1, \ldots, n$ and $bd=-\frac{2}{3}c,k_{0}=\frac{4}{3}nc, c\neq 0$, $m=0$.
\end{description}

In the following,  we discuss these two cases, respectively.

First, for the case (i), let $m=0$, then we have from \eqref{e10} and \eqref{e13}
$$F_{0}=a_{0}(\frac{1}{4}x^{4}-\frac{1}{2}z^{2})^{n}$$
and
$$F_{1}=-\frac{1}{3}(a+1)na_{0}(\frac{1}{4}x^{4}-\frac{1}{2}z^{2})^{n-1}x^{3}+na_{0}(\frac{1}{4}x^{4}-\frac{1}{2}z^{2})^{n-1}xy+\frac{1}{3}cna_{0}(\frac{1}{4}x^{4}-\frac{1}{2}z^{2})^{n-1}xz.$$
Then solving \eqref{e3} with $j=2$ for fixed $v$ and $w$ and substituting the transformation \eqref{e4},  we get
\begin{align*}
F_{2}=&(-\frac{1}{9}nc^{2}+\frac{1}{6}c^{2}+\frac{1}{2}a)na_{0}(\frac{1}{4}x^{4}-\frac{1}{2}z^{2})^{n-1}x^{2}+\frac{1}{2}n(n-1)a_{0}y^{2}(\frac{1}{4}x^{4}-\frac{1}{2}z^{2})^{n-2}x^{2}\\
&-\frac{1}{3}(a+1)n(n-1)a_{0}y(\frac{1}{4}x^{4}-\frac{1}{2}z^{2})^{n-2}x^{4}\\
&+[\frac{1}{18}(a+1)^{2}+\frac{1}{36}c^{2}]n(n-1)a_{0}(\frac{1}{4}x^{4}-\frac{1}{2}z^{2})^{n-2}x^{6}\\
&-\frac{1}{9}c(a+1)na_{0}(\frac{1}{4}x^{4}-\frac{1}{2}z^{2})^{n-1}z+\frac{1}{3}cn(n-1)a_{0}y(\frac{1}{4}x^{4}-\frac{1}{2}z^{2})^{n-2}x^{2}z\\
&-\frac{1}{9}c(a+1)n(n-1)a_{0}(\frac{1}{4}x^{4}-\frac{1}{2}z^{2})^{n-2}x^{4}z
+G_2(x,y,z),
\end{align*}
where
$G_2(x,y,z)=\overline{G}_2(v,w)=\sum\limits^{n}_{i=1}\overline{a}_{i}y^{2i-1}(\frac{1}{4}x^{4}-\frac{1}{2}z^{2})^{n-i},$
which is an arbitrary weight polynomial with degree $4n-2$.

Next, solving \eqref{e3} with $j=3$ for fixed $v$ and $w$,  we get
\begin{align*}
\overline F_{3}=&(\frac{1}{9}(a+1)na_{0}+\frac{1}{3}\overline{a}_{1})cvw^{n-1}\int\frac{1}{\sqrt{\frac{1}{2}u^{4}-2w}}du\\
&+\sum\limits^{n}_{i=2}(-\frac{2}{3}i+1)c\overline{a}_{i}v^{2i-1}w^{n-i}\int\frac{1}{\sqrt{\frac{1}{2}u^{4}-2w}}du\\
&+(\frac{2}{27}c^{3}-\frac{1}{9}a^{2}c
+\frac{1}{9}ac-\frac{1}{9}c-b)na_{0}w^{n-1}\int\frac{u^{2}}{\sqrt{\frac{1}{2}u^{4}-2w}}du+\overline{R}_{3}(u,v,w),
\end{align*}
where
\begin{align*}
\overline{R}_{3}(u,v,w)=&(\frac{2}{27}n-\frac{1}{9})c^{2}(a+1)na_{0}w^{n-1}u-\frac{1}{3}(a+1)n(n-1)a_{0}v^{2}w^{n-2}u\\
&+(-\frac{1}{9}nc^{2}+\frac{5}{18}c^{2}+\frac{1}{2}a+\frac{1}{9}(a+1)^{2})n(n-1)a_{0}vw^{n-2}u^{3}\\
&+\frac{1}{6}n(n-1)(n-2)a_{0}v^{3}w^{n-3}u^{3}\\
&+(\frac{1}{27}nc^{2}-\frac{1}{9}c^{2}-\frac{1}{6}a)(a+1)n(n-1)a_{0}w^{n-2}u^{5}\\
&-\frac{1}{6}(a+1)n(n-1)(n-2)a_{0}v^{2}w^{n-3}u^{5}\\
&+(\frac{1}{18}(a+1)^{2}+\frac{1}{36}c^{2})n(n-1)(n-2)a_{0}vw^{n-3}u^{7}\\
&-(\frac{1}{162}(a+1)^{2}+\frac{1}{108}c^{2})(a+1)n(n-1)(n-2)a_{0}w^{n-3}u^{9}\\
&-\frac{2}{9}c(a+1)n(n-1)a_{0}vw^{n-2}u\sqrt{\frac{1}{2}u^{4}-2w}\\
&+(-\frac{1}{81}nc^{2}+\frac{1}{27}(a+1)^{2}+\frac{7}{162}c^{2}+\frac{1}{6}a)cn(n-1)a_{0}w^{n-2}u^{3}\sqrt{\frac{1}{2}u^{4}-2w}\\
&+\frac{1}{6}cn(n-1)(n-2)a_{0}v^{2}w^{n-3}u^{3}\sqrt{\frac{1}{2}u^{4}-2w}\\
&-\frac{1}{9}(a+1)cn(n-1)(n-2)a_{0}vw^{n-3}u^{5}\sqrt{\frac{1}{2}u^{4}-2w}\\
&+(\frac{1}{54}(a+1)^{2}+\frac{1}{324}c^{2})cn(n-1)(n-2)a_{0}w^{n-3}u^{7}\sqrt{\frac{1}{2}u^{4}-2w}.
\end{align*}

The coefficients of $\int\frac{1}{\sqrt{\frac{1}{2}u^{4}-2w}}du$ and $\int\frac{u^{2}}{\sqrt{\frac{1}{2}u^{4}-2w}}du$ in $\overline{F}_{3}$ also should vanish. Thus,
\begin{equation}\label{e14}
\overline a_{1}=-\frac{1}{3}(a+1)na_{0}, \quad  \overline a_{i}=0, i=2, \ldots, n
\end{equation}
and
 \begin{equation}\label{e15}
 b=\frac{2}{27}c^{3}-\frac{1}{9}a^{2}c+\frac{1}{9}ac-\frac{1}{9}c.
 \end{equation}

Substituting the variable change \eqref{e4} and the relationship \eqref{e14}, \eqref{e15}, we have $F_{3}(x,y,z)=\overline R_{3}(u,v,w)$.

Continuing to solve \eqref{e3} with $j=4$,  we get
\begin{align*}
\overline F_{4}=(-\frac{1}{81}c^{3}-\frac{1}{27}a^{2}c+\frac{4}{27}ac-
\frac{1}{27}c)(a+1)na_{0}w^{n-1}\int\frac{u}{\sqrt{\frac{1}{2}u^{4}-2w}}du+\overline{R}_{4}(u,v,w),
\end{align*}
where $R_{4}(x,y,z)=\overline{R}_{4}(u,v,w)$ is a weight polynomial of weight degree $4n-4$. Then the coefficient of $\int\frac{u}{\sqrt{\frac{1}{2}u^{4}-2w}}du$ should vanish.
Hence we have
$$-\frac{1}{81}c^{3}-\frac{1}{27}a^{2}c+\frac{4}{27}ac-\frac{1}{27}c=0  \ \mbox{or} \ a=-1.$$

According to the above analysis, for Case (i) $F_{0}=a_{0}(\frac{1}{4}x^{4}-\frac{1}{2}z^{2})^{n}$, there are two different parameter conditions for system \eqref{e1}:
\begin{description}
\item{(1)} $a=-1, bd=-c, b=\frac{2}{27}c^{3}-\frac{1}{3}c,c\neq 0$.
\item{(2)}$-\frac{1}{81}c^{2}-\frac{1}{27}a^{2}+\frac{4}{27}a-\frac{1}{27}=0,  bd=-c,  b=\frac{2}{27}c^{3}-\frac{1}{9}a^{2}c+\frac{1}{9}ac-\frac{1}{9}c, c\neq 0.$
\end{description}

Secondly, we analyze case (ii) for $F_{0}$, it follows from \eqref{e10} and \eqref{e13} that
$$F_{0}=\sum\limits^{n}_{i=0}a_{i}y^{2i}(\frac{1}{4}x^{4}-\frac{1}{2}z)^{n-i}, $$
\begin{align*}
\overline F_{1}&=-\sum\limits^{n}_{i=0}\frac{1}{3}(a+1)(n-i)a_{i}v^{2i}w^{n-i-1}u^{3}+\sum\limits^{n}_{i=0}(n-i)a_{i}v^{2i+1}w^{n-i-1}u\\
&-\sum\limits^{n}_{i=0}\frac{1}{3}c(n-i)a_{i}v^{2i}w^{n-i-1}u\sqrt{\frac{1}{2}u^{4}-2w}.
\end{align*}

Solving \eqref{e3} with $j=2$ for fixed $v,w$, we get
\begin{align*}
\overline F_{2}=&\sum\limits^{n}_{i=0}[\frac{1}{3}c(n-i)a_{i}-2b(i+1)a_{i+1}]v^{2i+1}w^{n-i-1}\int\frac{u}{\sqrt{\frac{1}{2}u^{4}-2w}}du+\overline R_{2}(u,v,w),
\end{align*}
where $\overline R_{2}$ is a weight polynomial of weight degree $4n-2$. Because the coefficients of $\int\frac{u}{\sqrt{\frac{1}{2}u^{4}-2w}}du$ must vanish, we have $i=0,1,\ldots,n$
$$\frac{1}{3}c(n-i)a_{i}-2b(i+1)a_{i+1}=0.$$

Substituting the above equality, the form of $\overline F_{2}$ is as following
\begin{align*}
\overline F_{2}&=\sum\limits^{n}_{i=0}(-\frac{1}{9}(n-i)c^{2}+\frac{1}{6}c^{2}+\frac{a}{2})(n-i)a_{i}v^{2i}w^{n-i-1}u^{2}\\
&+\sum\limits^{n}_{i=0}\frac{1}{2}(n-i)(n-i-1)a_{i}v^{2i+2}w^{n-i-2}u^{2}\\
&-\sum\limits^{n}_{i=0}\frac{1}{3}(a+1)(n-i)(n-i-1)a_{i}v^{2i+1}w^{n-i-2}u^{4}\\
&+\sum\limits^{n}_{i=0}[\frac{1}{18}(a+1)^{2}+\frac{1}{36}c^{2}](n-i)(n-i-1)a_{i}v^{2i}w^{n-i-2}u^{6}\\
&-\sum\limits^{n}_{i=0}\frac{1}{9}c(a+1)(n-i)a_{i}v^{2i}w^{n-i-1}\sqrt{\frac{1}{2}u^{4}-2w}\\
&+\sum\limits^{n}_{i=0}\frac{1}{3}c(n-i)(n-i-1)a_{i}v^{2i+1}w^{n-i-2}u^{2}\sqrt{\frac{1}{2}u^{4}-2w}\\
&-\sum\limits^{n}_{i=0}\frac{1}{9}c(a+1)(n-i)(n-i-1)a_{i}v^{2i}w^{n-i-2}u^{4}\sqrt{\frac{1}{2}u^{4}-2w}\\
&+\sum\limits^{n}_{i=1}\overline{a}_{i}v^{2i-1}w^{n-i}.
\end{align*}

Substituting changes of variables \eqref{e4}, $F_{2}(x,y,z)=\overline F_{2}(u,v,w)$.

Next, solving \eqref{e3} with $j=3$ for fixed $v$ and $w$, we get
\begin{align*}
\overline F_{3}&=\sum\limits^{n}_{i=0}[\frac{1}{9}(a+1)(n-i)a_{i}+\frac{2}{3}\overline{a}_{i+1}]cv^{2i+1}w^{n-i-1}\int\frac{1}{\sqrt{\frac{1}{2}u^{4}-2w}}du\\
&+\sum\limits^{n}_{i=0}(\frac{2}{27}c^{3}-\frac{1}{9}a^{2}c+\frac{1}{9}ac-\frac{1}{9}c-b)(n-i)a_{i}v^{2i}w^{n-i-1}\int\frac{u^{2}}{\sqrt{\frac{1}{2}u^{4}-2w}}du\\
&+\overline R_{3}(u,v,w),
\end{align*}
where $\overline R_{3}$ is weight polynomial of weight degree $4n-3$. Because the coefficients of $\int\frac{1}{\sqrt{\frac{1}{2}u^{4}-2w}}du$ and $\int\frac{u^{2}}{\sqrt{\frac{1}{2}u^{4}-2w}}du$ must vanish, we have $i=0,1,\ldots,n$
\begin{equation*}
\overline a_{i+1}=-\frac{1}{6}(a+1)(n-i)a_{i},\quad b=\frac{2}{27}c^{3}-\frac{1}{9}a^{2}c+\frac{1}{9}ac-\frac{1}{9}c.
\end{equation*}

Substituting the above relationships of parameters in $\bar R_{3}(u,v,w)$, $\overline F_{3}$ becomes
\begin{align*}
\overline F_{3}&=\sum\limits^{n}_{i=0}(\frac{2}{27}(n-i)-\frac{1}{9})c^{2}(a+1)(n-i)a_{i}v^{2i}w^{n-i-1}u\\
&+\sum\limits^{n}_{i=0}(n-i)\overline a_{i}v^{2i}w^{n-i-1}u\\
&+\sum\limits^{n}_{i=0}(-\frac{1}{9}(n-i)c^{2}+\frac{17}{54}c^{2}+\frac{1}{2}a)(n-i)(n-i-1)a_{i}v^{2i+1}w^{n-i-2}u^{3}\\
&+\sum\limits^{n}_{i=0}\frac{1}{6}(n-i)(n-i-1)(n-i-2)a_{i}v^{2i+3}w^{n-i-3}u^{3}\\
&-\sum\limits^{n}_{i=0}\frac{2}{9}bic(n-i)a_{i}v^{2i-1}w^{n-i-1}u^{3}\\
&-\sum\limits^{n}_{i=0}\frac{1}{3}(a+1)(n-i)\overline a_{i}v^{2i-1}w^{n-i-1}u^{3}\\
&+\sum\limits^{n}_{i=0}(\frac{1}{27}(n-i)c^{2}-\frac{1}{9}c^{2}-\frac{1}{6}a)(a+1)(n-i)(n-i-1)a_{i}v^{2i}w^{n-i-2}u^{5}\\
&-\sum\limits^{n}_{i=0}\frac{1}{6}(a+1)(n-i)(n-i-1)(n-i-2)a_{i}v^{2i+2}w^{n-i-3}u^{5}\\
&+\sum\limits^{n}_{i=0}(\frac{1}{18}(a+1)^{2}+\frac{1}{36}c^{2})(n-i)(n-i-1)(n-i-2)a_{i}v^{2i+1}w^{n-i-3}u^{7}\\
&-\sum\limits^{n}_{i=0}(\frac{1}{162}(a+1)^{2}+\frac{1}{108}c^{2})(a+1)(n-i)(n-i-1)(n-i-2)a_{i}v^{2i}w^{n-i-3}u^{9}\\
&-\sum\limits^{n}_{i=0}\frac{5}{27}c(a+1)(n-i)(n-i-1)a_{i}v^{2i+1}w^{n-i-2}u\sqrt{\frac{1}{2}u^{4}-2w}\\
&+\sum\limits^{n}_{i=0}\frac{4}{9}bi(a+1)(n-i)a_{i}v^{2i-1}w^{n-i-1}u\sqrt{\frac{1}{2}u^{4}-2w}\\
&+\sum\limits^{n}_{i=0}\frac{1}{3}c(n-i)\overline a_{i}v^{2i-1}w^{n-i-1}u\sqrt{\frac{1}{2}u^{4}-2w}\\
&+\sum\limits^{n}_{i=0}(-\frac{1}{81}(n-i)c^{2}+\frac{1}{27}(a+1)^{2}+\frac{7}{162}c^{2}+\frac{1}{6}a)\\
&c(n-i)(n-i-1)a_{i}v^{2i}w^{n-i-2}u^{3}\sqrt{\frac{1}{2}u^{4}-2w}\\
&+\sum\limits^{n}_{i=0}\frac{1}{6}c(n-i)(n-i-1)(n-i-2)a_{i}v^{2i+2}w^{n-i-3}u^{3}\sqrt{\frac{1}{2}u^{4}-2w}\\
&-\sum\limits^{n}_{i=0}\frac{1}{9}c(a+1)(n-i)(n-i-1)(n-i-2)a_{i}v^{2i+1}w^{n-i-3}u^{5}\sqrt{\frac{1}{2}u^{4}-2w}\\
&+\sum\limits^{n}_{i=0}\frac{1}{9}(\frac{1}{6}(a+1)^{2}+\frac{1}{36}c^{2})c(n-i)(n-i-1)(n-i-2)a_{i}v^{2i}w^{n-i-3}u^{7}\sqrt{\frac{1}{2}u^{4}-2w}.
\end{align*}

Moreover, solving \eqref{e3} with $j=4$ for fixed $v$ and $w$, we get
\begin{align*}
\overline F_{4}=\sum\limits^{n}_{i=0}(-\frac{1}{27}c^{3}-\frac{1}{54}a^{2}c+\frac{7}{54}ac-\frac{1}{54}c)(a+1)(n-i) a_{i}v^{2i}w^{n-i-1}\int\frac{u}{\sqrt{\frac{1}{2}u^{4}-2w}}du+\overline R_{4}
\end{align*}
where $\overline R_{4}$ is weight polynomial of weight degree $4n-4$. The coefficient of $\int\frac{u}{\sqrt{\frac{1}{2}u^{4}-2w}}du$ must vanish. Then we have
$$-\frac{1}{27}c^{3}-\frac{1}{54}a^{2}c+\frac{7}{54}ac-\frac{1}{54}c=0\ \mbox{or} \ a=-1.$$

According to the above analysis, for Case (ii) $F_{0}=\sum\limits^{n}_{i=0}a_{i}y^{2i}(\frac{1}{4}x^{4}-\frac{1}{2}z)^{n-i}$, there are also two different conditions of parameters of system \eqref{e1}:
\begin{description}
\item{(3)} $a=-1, bd=-\frac{2}{3}c, b=\frac{2}{27}c^{3}-\frac{1}{3}c,c\neq 0$;
\item{(4)}$-\frac{1}{81}c^{2}-\frac{1}{27}a^{2}+\frac{4}{27}a-\frac{1}{27}=0,  bd=-\frac{2}{3}c,  b=\frac{2}{27}c^{3}-\frac{1}{9}a^{2}c+\frac{1}{9}ac-\frac{1}{9}c, c\neq 0$.
\end{description}
Above all, the proof of Lemma \eqref{l3'} has been completed.
\end{proof}

\begin{remark}\label{r3}
According to Lemma \ref{l3'}, the cofactor of system \eqref{e1} must be $\frac{4}{3}nc$, where $n$ is in the form of $F_0$ in \eqref{e10} .
\end{remark}

Assume that $\phi(x,y,z)$ is an irreducible Darboux polynomial of system \eqref{e1}, then $\Phi(x,y,z)=\phi^{n}(x,y,z)$ can be expanded on the basis of characteristic curves \eqref{E3}, which can be denoted as $\Phi(x,y,z)=\Phi_{0}(x,y,z)+\Phi_{1}(x,y,z)+\ldots+\Phi_{l}(x,y,z)$, where $\Phi_{j}(x,y,z)$, $j=0,1,\ldots,l$ are weight polynomials with weight exponents (1,2,2), weight degree of $\Phi_{j}(x,y,z)$ is $l-j$. Due to this fact, we have the following result.

\begin{lemma}\label{l4}
Assume that $\phi(x,y,z)$ is an irreducible Darboux polynomial of system \eqref{e1} with cofactor $\overline{k}$. Let $\Phi(x,y,z)=\phi^{n}(x,y,z)=\Phi_0+\Phi_1+\ldots+\Phi_l$. If $\Phi_{0}=F_{0}$, where $F_{0}$ is in the form of \eqref{e10}, $\Phi(x,y,z)$ and $f(x,y,z)$ satisfy the same conditions of parameters and have the same cofactor, then $f(x,y,z)|_{m=0}=\Phi(x,y,z)$, where $f=F|_{\alpha=1}=F_{0}+F_{1}+\ldots+F_{l}$.
\end{lemma}
\begin{proof}
According to Remark \ref{r3}, the cofactor of $f|_{m=0}$ is $\frac{4}{3}nc$. $\Phi(x,y,z)$ is a Darboux polynomial with cofactor $n\overline{k}$, then $\overline{k}=\frac{4}{3}c$.

Assume that $f|_{m=0}(x,y,z)\neq\Phi(x,y,z)$, then $\Phi-f|_{m=0}$ is a Darboux polynomial of system \eqref{e1} with cofactor $\frac{4}{3}nc$. We also expand $\Phi-f|_{m=0}$ on the basis of \eqref{E3}, then we have $\Phi-f|_{m=0}=\eta_{0}+\eta_{1}+\ldots+\eta_{t}$, where $\eta_{j}(x,y,z),j=0,1,\ldots,t$ are weight polynomials with weight exponents (1,2,2), the weight degree of $\eta_{j}(x,y,z)$ is $t-j$.

If $F_{0}$ is in the form of \eqref{e10}, it follows from Remark \ref{r2} and $\Phi_{0}=F_{0}$ that $l=4n$ and $t=4p,p\leq n-1$. Thus, according to Remark \ref{r3}, the cofactor of $\Phi-f|_{m=0}$ is $\frac{4}{3}pc$. Because the cofactor for one Darboux polynomial is unique, Then we have $f(x,y,z)=\Phi(x,y,z)$.
\end{proof}

\begin{lemma}\label{l3}
If $F_{0}$ is in the form \eqref{e10}, system \eqref{e1} have irreducible Darboux polynomials with nonzero cofactor $\frac{4}{3}c$ and $c\not=0$ shown in Table 1.
\end{lemma}
\begin{proof}
According to Lemma \ref{l3'}, we have that the parameters in system \eqref{e1} should satisfy conditions $(1),(2),(3),(4)$ which are listed in the lemma.

If parameters of system \eqref{e1} satisfy condition $(1)$, $\phi_{1}$ is a Darboux polynomial of system \eqref{e1} with the cofactor $\frac{4}{3}c$ by a simple calculation,
where $$\phi_{1}(x,y,z)=\frac{1}{2}x^{4}-z^{2}+2xy+\frac{2}{3}cxz+(\frac{1}{9}c^{2}-1)x^{2}.$$

Let $\Phi(x,y,z)=\phi_{1}^{n}(x,y,z)$, and expand $\Phi(x,y,z)=\Phi_{0}+\Phi_{1}+\cdots+\Phi_{4n}$, where $\Phi_{j}(x.y,z)$ is a weight homogeneous polynomial of weight degree $4n-j$, weight exponents of $\Phi_{j}$ is $(1,2,2)$. It is easy to check $\Phi_{0}=F_{0}$. Then, it follows from Lemma \ref{l4} that $f|_{m=0}=\Phi(x,y,z).$ It is to say, $F_{0}+F_{1}+\cdots+F_{4n}=\phi_{1}^{n}$. Thus,$\phi_{1}(x,y,z)$ is an irreducible Darboux polynomial of system \eqref{e1} with cofactor $\frac{4}{3}c$ under condition $(1)$ of parameters.

If the parameters of system \eqref{e1} satisfy condition $(2),(3),(4)$, $\phi_{2},\phi_{3}$ and $\phi_4$ are Darboux polynomials of system \eqref{e1} with the cofactor $\frac{4}{3}c$ by a simple calculation,
where $$\phi_{2}=\frac{1}{2}x^{4}-z^{2}-\frac{2}{3}(a+1)x^{3}+2xy+\frac{2}{3}cxz-\frac{2}{9}c(a+1)z+(\frac{1}{9}c^{2}+a)x^{2}-\frac{2}{3}(a+1)y-\frac{2}{27}c^{2}(a+1)x$$
$$\phi_{3}=\frac{1}{2}x^{4}-z^{2}-\frac{1}{2}dy^{2}+2xy+\frac{2}{3}cxz+(\frac{1}{9}c^{2}-1)x^{2},$$
\begin{align*}
\phi_{4}&=\frac{1}{2}x^{4}-z^{2}-\frac{1}{2}dy^{2}-\frac{2}{3}(a+1)x^{3}+2xy+\frac{2}{3}cxz\\
&-\frac{2}{9}c(a+1)z+(\frac{1}{9}c^{2}+a)x^{2}-\frac{1}{3}(a+1)y-\frac{2}{27}c^{2}(a+1)x.
\end{align*}

It follows from Lemma \ref{l3'} and \ref{l4} that $\phi_{2},\phi_{3}$ and $\phi_4$ are irreducible Darboux polynomials of system\eqref{e1} with cofactor $\frac{4}{3}c$ under condition $(2),(3),(4)$ of parameters. The proof of these conditions are the same as the proof of condition $(1)$.
\end{proof}

\subsection{$k_{0}=0$}

In this part,  we classify Darboux polynomials with cofactor zero. In fact, Darboux polynomial with cofactor zero is the polynomial integral of the system. The method we used in this part is the same as the one which is used in the former section. As we know, no matter the cofactor $k_{0}$ is zero or not, $F_{0}$ should be in the form of \eqref{e9} and \eqref{e10}. When $k_{0}=0$ and Darboux polynomials are not only depend on variable $y$, we have the following result.

\begin{lemma}\label{l5}
If $k_{0}=0$, and Darboux polynomials of system \eqref{e1} are not only depend on variable $y$,  then the parameters $b$ and $c$ in system \eqref{e1} are equal to zero.
\end{lemma}
\begin{proof}
If $F_{0}$ is in the form \eqref{e9}, the equality \eqref{E0} in the proof of Lemma \ref{l2} holds whether $k_{0}=0$ or not. It follows from \eqref{E0} that $ba_{1}=0$.

If $b\neq0$, then $a_{1}=0$, it is from \eqref{E0} that $a_{i}=0,i=1,\ldots,n$, then $F_{0}=0$, which contradicts with Remark \ref{r1}. So $b=0$.

If $b=0$, \eqref{E0} becomes $$\sum\limits^{n-1}_{i=1}c(n-i)a_{i}v^{2i}w^{n-i-1}=0.$$
 If there exists $a_i\neq0,i=1,2,\ldots,n-1$, then $c=0$; If $a_n\neq0,a_i=0,i=1,2,\ldots,n-1$, then $F_0=a_ny^{2n-1}$, and the polynomial integral is only depend on variable $y$. Thus, if $b=0$, then $c=0$.

 If $F_{0}$ is in the form \eqref{e10}, the equalities \eqref{e12} and \eqref{e19} in the proof of Lemma \ref{l3'} hold whether $k_{0}=0$ or not. We discuss this situation also in the following cases:
\begin{description}
\item{(i)} $bd=-c$

The equality $\frac{2}{3}(2n+i)c=0$ holds from \eqref{e12}, then $c=0$. Substituting $bd=-c$ into \eqref{e19}, we have $2bia_{i}=0,i=1,2,\ldots,n$. If there exits $a_{i}\neq0,i=1,2,\ldots,n$, then $b=0$. If $a_{0}\neq0,a_{i}=0,i=1,2,\ldots,n$, then $F_{0}=a_{0}(\frac{1}{4}x^{4}-\frac{1}{2}z^{2})^{n}$, and it follows from \eqref{e15} that $b=0$.

\item{(ii)} $bd\neq-c$

If $b=0$, then $c\neq0$, it follows from \eqref{e19} that $c(n-i)a_{i}=0,i=0,1,\ldots,n$, then $a_{i}=0,i=0,1,\ldots,n-1$. It follows from Remark \ref{r1} that $a_{n}\neq0,a_{i}=0,i=0,1,\ldots,n-1$. Let $m=0$, then $F_{0}=a_{n}y^{2n}$ and $f|_{m=0}=y^{2n}$. It contradicts with $f$ is not only depend on $y$.

If $b\neq0,c=0$, then $bd\neq0$. From \eqref{e12} we have $-2ibd=0$, then $a_{0}\neq0,a_{i}=0,i=1,2,\ldots,n$. and $F_{0}=a_{0}(\frac{1}{4}x^{4}-\frac{1}{2}z^{2})^{n}$.
Substituting $F_{0}$ and $c=0,k_{0}=0$ in \eqref{e3}, when $j=2$ we have $bd=0$. which contradicts with $bd \neq 0$.

If $b\neq0,c\neq0$, it follows from \eqref{e19} that $a_{i}\neq0,i=0,1,\ldots,n$. Because the cofactor of a Darboux polynomial  is unique, we have $k_{0}=\frac{4}{3}nc=0,bd=-\frac{2}{3}c$, which contradicts with $c\neq0$.

So, the relation $bd \neq -c$ doesn't hold.
\end{description}

The lemma has been proved.
\end{proof}

In the following, we just need to consider the situation when $b=c=0$. And it is obvious that $y$ is an irreducible Darboux polynomial of system \eqref{e1} when $k_{0}=0$.

If $k_{0}=0,b=c=0$, \eqref{e3} becomes
\begin{equation}\label{e25}
\begin{array}{ll}
L[F_{1}]&=[(a+1)x^{2}-y]\frac{\partial{F_{0}}}{\partial{z}},\\
L[F_{j}]&=[(a+1)x^{2}-y]\frac{\partial{F_{j-1}}}{\partial{z}}-ax\frac{\partial{F_{j-2}}}{\partial{z}},\qquad\hfill j=2, 3, \ldots, l+3,
\end{array}
\end{equation}
where $F_{j}=0,j>l$.

 If $b=c=0$, we consider $F_{0}^{\ast}=a_{0}(\frac{1}{4}x^{4}-\frac{1}{2}z^{2})^{n}$. Let $m=0$, substituting $F_{0}$ into equations of \eqref{e25} and solving the the differential equations, we can obtain the forms of $F_{j}^{\ast},j=1,2,\ldots,4n$. The calculation procedure is the same as the one which is used in the proof when $k_0 \neq 0$.

If $m=0,b=c=0$,
$$\phi_{5}=\frac{1}{4}x^{4}-\frac{1}{2}z^{2}-\frac{1}{3}(a+1)x^{3}+xy+\frac{1}{2}ax^{2}$$ is a polynomial integral of system \eqref{e1}.

 It follows from Lemma \ref{l3'} and \ref{l4} that $\phi_{5}^{n}(x,y,z)=F_{0}^{*}+F_{1}^{*}+\ldots+F_{4n}^{*}$. The proof is the same as that in the proof of Lemma \ref{r3}.

 Thus, $\frac{1}{4}x^{4}-\frac{1}{2}z^{2}-\frac{1}{3}(a+1)x^{3}+xy+\frac{1}{2}ax^{2}$ and $y$ are polynomial integrals of system \eqref{e1}.

When $F_{0}$ is in the general form of \eqref{e9} and \eqref{e10}, by calculating equations in \eqref{e25}, we obtain a polynomial integral of system \eqref{e1} denoted by $F(x,y,z)$. The relationship between $F(x,y,z)$ and $y,\phi_5$  is shown in the following Lemma.
\begin{lemma}\label{l7}
If $m=0$, $F_{0}=y^{p}(\frac{1}{4}x^{4}-\frac{1}{2}z^{2})^{q}, p,q\in\mathbb{N}$, then the polynomial integral calculated by \eqref{e25} with $F_{0}$ is $y^{p}\phi_{5}^{q}$.
\end{lemma}
\begin{proof}
Obviously, if $\tilde F_{0}=(\frac{1}{4}x^{4}-\frac{1}{2}z^{2})^{q}$, then the Darboux polynomial calculated by $\tilde F_{0}$ is $\phi_{5}^{q}.$

Assume that when $\tilde F_{0}=(\frac{1}{4}x^{4}-\frac{1}{2}z^{2})^{q}$, solving equations in \eqref{e25} we have $\tilde F_{j},j=1,2,\ldots,4q.$

Notice that if $m=0$, differential equations in \eqref{e25} are linear and don't have relationship with the partial derivatives of variable $y$. Thus, when $F_{0}=y^{p}(\frac{1}{4}x^{4}-\frac{1}{2}z^{2})^{q}$, we have $F_{j}=y^{p}\tilde F_{j},j=1,2,\ldots,4q.$

Thus, the polynomial integral calculated by $F_{0}=y^{p}(\frac{1}{4}x^{4}-\frac{1}{2}z^{2})^{q}$ is $y^{p}\phi_{5}^{q}$.
\end{proof}

According to Lemma \ref{l7} and the linearity of equations in \eqref{e25}, we have the conclusion that if $m=0$, and $F_{0}$ is in the form of \eqref{e9} or \eqref{e10}, the polynomial integral can be generated by $y$ and $\phi_{5}$.

\textbf{[The proof of Theorem 1]}
According to Lemma \ref{l3} and the conclusion we obtained when $k_0=0$, we have a complete classification of invariant algebraic surfaces of F-N system \eqref{e1}.

\subsection{Proof of the corollary}
\begin{proof}
\begin{description}
\item{(a)}\ It is obvious from Theorem \ref{t1}.
\item{(b)}\ It follows from theorem \ref{t1} that system \eqref{e1} has two functional independent polynomial integrals, then F-N system \eqref{e1} is polynomial integrable.
\end{description}
\end{proof}

\begin{appendices}

The following formulas are used in the proof of Theorem \ref{t1}:\\
Denoting E=EllipticE$(\frac{1}{2}u\sqrt{-\frac{2}{\sqrt{w}}},  I)$,  F=EllipticF$(\frac{1}{2}u\sqrt{-\frac{2}{\sqrt{w}}}, I)$.
$$\int\frac{1}{\sqrt{\frac{1}{2}u^{4}-2w}}du=\frac{\sqrt{\frac{4+2u^{2}}{\sqrt{w}}}\sqrt{\frac{\sqrt{4-2u^{2}}}{\sqrt{w}}}F}{\sqrt{-\frac{2}{\sqrt{w}}}\sqrt{2u^{4}-8w}};$$
$$\int\frac{u^{2}}{\sqrt{\frac{1}{2}u^{4}-2w}}du=2\sqrt{w}\frac{\sqrt{\frac{4+2u^{2}}{\sqrt{w}}}\sqrt{\frac{\sqrt{4-2u^{2}}}{\sqrt{w}}}(F-E)}{\sqrt{-\frac{2}{\sqrt{w}}}\sqrt{2u^{4}-8w}};$$
$$\int\frac{u}{\sqrt{\frac{1}{2}u^{4}-2w}}du=\frac{\sqrt{2}}{2}\ln{(\sqrt{2}u^{2}+\sqrt{2u^{4}-8w})}.$$
Next, we denote $A=\int\frac{1}{\sqrt{\frac{1}{2}u^{4}-2w}}du,B=\int\frac{u^2}{\sqrt{\frac{1}{2}u^{4}-2w}}du,C=\int\frac{u}{\sqrt{\frac{1}{2}u^{4}-2w}}du$, then other integrals in the calculation can be expressed by $A,B$ and $C$. The formulas are presented as following:
$$\int\sqrt{\frac{1}{2}u^{4}-2w}du=\frac{1}{3}u\sqrt{\frac{1}{2}u^{4}-2w}-\frac{4}{3}wA;$$
$$\int u\sqrt{\frac{1}{2}u^{4}-2w}du=\frac{1}{4}u^{2}\sqrt{\frac{1}{2}u^{4}-2w}-wC;$$
$$\int\frac{u^{4}}{\sqrt{\frac{1}{2}u^{4}-2w}}du=\frac{2}{3}u\sqrt{\frac{1}{2}u^{4}-2w}+\frac{4}{3}wA;$$
$$\int\frac{u^{5}}{\sqrt{\frac{1}{2}u^{4}-2w}}du=\frac{1}{2}u^{2}\sqrt{\frac{1}{2}u^{4}-2w}+2wC;$$
$$\int\frac{u^{6}}{\sqrt{\frac{1}{2}u^{4}-2w}}du=\frac{2}{5}u^{3}\sqrt{\frac{1}{2}u^{4}-2w}+\frac{12}{5}wB;$$
$$\int\frac{u^{9}}{\sqrt{\frac{1}{2}u^{4}-2w}}du=(\frac{1}{4}u^{6}+\frac{3}{2}w^{2}u)\sqrt{\frac{1}{2}u^{4}-2w}+6w^{2}C;$$
$$\int u^{2}\sqrt{\frac{1}{2}u^{4}-2w}du=\frac{1}{5}u^{3}\sqrt{\frac{1}{2}u^{4}-2w}-\frac{4}{5}wB;$$
$$\int u^{4}\sqrt{\frac{1}{2}u^{4}-2w}du=(\frac{1}{7}u^{5}-\frac{8}{21}wu)\sqrt{\frac{1}{2}u^{4}-2w}-\frac{16}{21}w^{2}A;$$
$$\int u^{5}\sqrt{\frac{1}{2}u^{4}-2w}du=(\frac{1}{8}u^{6}-\frac{1}{4}wu^{2})\sqrt{\frac{1}{2}u^{4}-2w}-w^{2}C;$$
$$\int u^{6}\sqrt{\frac{1}{2}u^{4}-2w}du=(\frac{1}{9}u^{7}-\frac{8}{45}wu^{3})\sqrt{\frac{1}{2}u^{4}-2w}-\frac{16}{15}w^{2}B;$$
$$\int u^{9}\sqrt{\frac{1}{2}u^{4}-2w}du=(\frac{1}{12}u^{10}-\frac{1}{12}wu^{6}-\frac{1}{2}w^{2}u^{2})\sqrt{\frac{1}{2}u^{4}-2w}-2w^{3}A;$$
$$\int(\sqrt{2}u^{2}+2\sqrt{\frac{1}{2}u^{4}-2w})^{-1}du=\frac{\sqrt{2}}{24}u^{3}w^{-1}-\frac{1}{12}uw^{-1}\sqrt{\frac{1}{2}u^{4}-2w}-\frac{8}{3}A;$$
$$\int(\sqrt{2}u^{2}+2\sqrt{\frac{1}{2}u^{4}-2w})^{-1}\frac{1}{\sqrt{\frac{1}{2}u^{4}-2w}}du=-\frac{1}{4}uw^{-1}-\frac{\sqrt{2}}{8}w^{-1}B;$$
$$\int(\sqrt{2}u^{2}+2\sqrt{\frac{1}{2}u^{4}-2w})^{-1}\frac{u^{2}}{\sqrt{\frac{1}{2}u^{4}-2w}}du=-\frac{1}{12}u^{3}w^{-1}+\frac{\sqrt{2}}{12}uw^{-1}\sqrt{\frac{1}{2}u^{4}-2w}+\frac{\sqrt{2}}{6}A.$$
\end{appendices}

\section*{Acknowledgments}

The second author is partially supported by NNSF of China grant number 11431008 and NSF of Shanghai grant number  15ZR1423700.

\end{document}